\documentclass[reqno, a4paper, 10pt]{amsart}
\usepackage{amsmath,amsthm,amssymb,amscd,amstext,amsopn,amsxtra,amsfonts, bbm, mathrsfs, eucal}
\usepackage{color}

\usepackage[all]{xy}
\SelectTips{eu}{}
\entrymodifiers={+!!<0pt,\fontdimen22\textfont2>}

\usepackage[pdftex]{hyperref}
\hypersetup{linktocpage}
\usepackage{url}

\swapnumbers
\theoremstyle{plain}
\newtheorem{thm}[subsection]{Theorem}
\newtheorem{lemma}[subsection]{Lemma}
\newtheorem{prop}[subsection]{Proposition}
\newtheorem{cor}[subsection]{Corollary}

\theoremstyle{remark}

\theoremstyle{definition}
\newtheorem{example}[subsection]{Example}
\newtheorem{remark}[subsection]{Remark}
\newtheorem{assumption}[subsection]{Assumption}
\newtheorem*{warning}{Warning}

\numberwithin{equation}{subsection}

\def\e{\mathbbm{1}}

\def\cA{\mathcal{A}}
\def\cB{\mathcal{B}}
\def\cC{\mathcal{C}}
\def\cD{\mathcal{D}}

\def\cH{\mathcal{H}}

\def\cL{\mathcal{L}}

\def\cO{\mathcal{O}}

\def\cT{\mathcal{T}}

\def\cZ{\mathcal{Z}}

\def\11{\mathbf{1}}

\def\CC{\mathbf{C}} 
\def\DD{\mathbf{d}}

\def\TT{\mathbf{T}} 
 
\def\VV{\mathbb{V}}

\def\ZZ{\mathbf{Z}}

\def\fb{\mathfrak{b}}

\def\fg{\mathfrak{g}} 
\def\fh{\mathfrak{h}}

\def\fz{\mathfrak{z}}

\def\coev{\mathrm{coev}}

\def\Comp{\mathrm{Kom}}

\def\Db{\mathrm{D^b}}

\def\dim{\mathrm{dim}} 
\def\End{\mathrm{End}}
\def\Ext{\mathrm{Ext}}
\def\ev{\mathrm{ev}}

\def\gmof{\mathrm{-gmof}}

\def\Hom{\mathrm{Hom}}
\def\HOM{\mathscr{H}\!om}

\def\id{\mathrm{id}}

\def\mof{\mathrm{-mof}}

\newcommand{\mapright}[1]{\xrightarrow{#1}}

\newcommand{\Verma}[1]{M_{#1}}
\newcommand{\coVerma}[1]{M^{\vee}_{#1}}
\newcommand{\simple}[1]{L_{#1}}
\newcommand{\projective}[1]{P_{#1}}

\newcommand{\gVerma}[1]{\mathtt{M}_{#1}}
\newcommand{\gcoVerma}[1]{\mathtt{M}_{#1}^{\vee}}
\newcommand{\gsimple}[1]{\mathtt{L}_{#1}}
\newcommand{\gprojective}[1]{\mathtt{P}_{#1}}
\newcommand{\gtilting}[1]{\mathtt{D}_{#1}}

\makeatletter
\renewcommand{\@makefnmark}{\mbox{\textsuperscript{}}}
\makeatother

\title{A note on Hecke patterns in category $\cO$}
\author{R. Virk}
\address{Department of Mathematics\\
University of California\\
Davis, CA 95616}
\email{virk@math.ucdavis.edu}

\begin{document}

\maketitle
\setcounter{tocdepth}{1}
\tableofcontents
\section{Introduction}
The purpose of this document is to study a family of auto-equivalences of the derived category of the principal block of the BGG-category $\cO$. In the geometric setting (i.e., perverse sheaves or $D$-modules on the flag variety) it is well known that much of the information of interest to representation theory is encoded in the convolution structure on the relevant categories of sheaves/$D$-modules. This is the theory of the geometric Hecke algebra and `Hecke patterns', see \cite{B}, \cite{BBM}, \cite{BD}, \cite{BeGi}, \cite{L}, \cite{LV}, \cite{T}, \cite{So10}. The equivalences studied in this note correspond to the `standard generators' of the Hecke algebra. One of the goals is to show that many of the results regarding category $\cO$ in the literature are very natural from this point of view: namely that of category $\cO$ as a `reasonably faithtful module' for the Hecke algebra (see \cite{So10}). Our approach is algebraic - perverse sheaves and the geometry of the flag variety are notably absent in our arguments. In the conclusion we do explain how stronger results can be achieved using an additional assumption (Assumption \ref{koszul}). However, as far as I am aware, the only known proof of this assumption is geometric.

Let me now describe the contents of this document and indicate the main results. In \S\ref{s:not}-\S\ref{s:complexesoffunctors} we set up some homological algebra that culminates in \S\ref{s:genconst} in the form of Thm.\ \ref{mainthm} which is originally due to Rickard \cite[Thm.\ 2.1]{Ri} (also see \cite[\S 2.2.3]{Ro}, \cite[Lemma 4.1.1.]{ABG}, \cite[Thm.\ 7.3.16]{Vo}).

In \S\ref{s:catO} we introduce the BGG category $\cO$ and following \cite[\S2.10]{Ja} consider translation and wall crossing functors. Thm.\ \ref{mainthm} is exploited to construct the aforementioned derived auto-equivalences of the principal block of $\cO$ (Prop.\ \ref{dequivO}). Using these we give a quick proof of `Bott's Theorem' \cite[Thm.\ 15]{Bott} in Thm.\ \ref{bottsthm}.

The constructed derived equivalences satisfy the braid relations, in our setting this is due to Rouquier \cite[Thm.\ 4.4]{Ro}. In \S\ref{s:lastnongraded} we exploit the braid relations to show that there is a derived auto-equivalence that switches tilting modules with projective modules (Thm.\ \ref{switchtilting}). Our proof is formally the same as that of \cite[Prop.\ 2.3]{BBM} (also see \cite[Thm.\ 8]{StM}). In fact, the auto-equivalences considered in this document are Koszul dual (in the sense of \cite{BGS}) to the Radon transforms of \cite{BBM}. In Cor.\ \ref{tiltingcharformula} and Cor.\ \ref{ringelselfduality} we recover Soergel's character formula for tilting modules \cite[Thm.\ 6.7]{So98} and the Ringel self duality of the principal block (implicit in \cite{So98}). It should be pointed out that although Soergel doesn't explicitly construct a derived equivalence in \cite{So98} (he works with categories of modules with Verma/dual Verma flags), the derived functor of Arkhipov's twisting functor considered by him is a derived equivalence. In fact, (derived) twisting functors correspond to the Radon transforms of \cite{BBM} and so our approach is essentially Koszul dual to Soergel's.

In \S\ref{s:gO}, following Soergel and Stroppel, we considered graded category $\cO$. This section makes heavy use of \cite{St}. Proceeding as in the non-graded case we construct derived auto-equivalences in this setting and prove graded analogues of the results in the previous sections. In particular, we direct the reader to Thm.\ \ref{dequivOgraded} and \S\ref{s:gtilting}.

Finally, in \S\ref{s:kl}, we explain the connection between our auto-equivalences and Kazhdan-Lusztig theory. The main results are Thm.\ \ref{klequivtilt} and Thm.\ \ref{klconj}. Assumption \ref{koszul} and Thm.\ \ref{klconj} are the only results in this note that depend on geometric results. 

\subsection*{Acknowledgments}
I am grateful to W.\ Soergel for some extremely helpful correspondence. I also thank A.\ Ram for convincing me that this note needed to be written, without his encouragement this document would have never seen the light of day. Part of this document was written while I was a graduate student at the University of Wisconsin-Madison and I thank the department there for its support. This work is partially supported by the NSF grant DMS-0652641.

\section{Notations and conventions}\label{s:not}
\subsection{}Functors between additive categories will be assumed to be additive.

\subsection{}The terms `functorial', `natural' and `canonical' will be used as synonyms for `a morphism of functors'.

\subsection{}If $\cA$ is an additive category, we write $\Comp(\cA)$ for the category of complexes in $\cA$. If $\cA$ is abelian, we write $\Db(\cA)$ for the bounded derived category of $\cA$.

\subsection{}When working with triangulated categories we denote the shift functor by $[1]$. Distinguished triangles $X\to Y \to Z\to X[1]$ will often be written as $X\to Y\to Z\leadsto$.

\subsection{}Let $\cT$ be a triangulated category. We say that an object $X\in\cT$ is filtered by objects $Y_1,\ldots, Y_n$ if there exists a sequence of objects $0=X_0, X_1, \ldots, X_n=X$ and distinguished triangles $X_{i-1}\to X_i \to Y_i \leadsto$. We will often use this notion in the following situation: let $H$ be a cohomological functor on $\cT$. Let $X, X_i, Y_i$ be as above. Assume that $H(Y_i[m])=0$ for all $m\in \ZZ$ and all $i$. Then, proceeding by induction it follows that $H(X[m])=0$ for all $m\in \ZZ$.

\subsection{}If $\cA$ is an abelian or triangulated category, we write $K_0(\cA)$ for the Grothendieck group of $\cA$. If $\cA$ is abelian, then $K_0(\cA)$ and $K_0(\Db(\cA))$ are canonically isomorphic and we take the liberty of identifying them with each other.

\section{Reminders on adjoint functors}\label{s:1}
\subsection{}\label{s:adjunctions}Let $f_*\colon\cA\to\cB$ and $f^*\colon\cB\to\cA$ be functors. An adjunction $(f^*,f_*)$ between $f^*$ and $f_*$ is the data of two natural transformations $\varepsilon\colon f^*f_*\to\id_{\cA}$ and $\eta:\id_{\cB}\to f_*f^*$ such that the compositions 
\begin{equation}\label{eq:unitcounit} f_*\mapright{\eta \e_{f_*}} f_*f^*f_*\mapright{\e_{f_*}\varepsilon}f_* \qquad \mbox{and} \qquad f^*\mapright{\e_{f^*} \eta}f^*f_*f^*\mapright{\varepsilon \e_{f^*}} f^*\end{equation}
are equal to the identity on $f_*$ and $f^*$, respectively. The morphisms $\eta$ and $\varepsilon$ are the unit and counit of the adjunction respectively.
An adjunction gives an isomorphism, functorial in $A\in\cA$ and $B\in\cB$:
\[\alpha_{A,B}\colon \Hom_{\cA}(f^*B, A){\mapright \sim} \Hom_{\cB}(B, f_*A), \qquad
\phi\mapsto \e_{f_*}\phi\circ\eta_B.
\]
The inverse is given by $\psi\mapsto \varepsilon_A\circ \e_{f^*}\psi$.
Conversely, a functorial isomorphism $\alpha_{A,B}$ as above provides an adjunction $(f^*,f_*)$. Namely, set $\varepsilon_A=\alpha^{-1}_{A,f_*A}(\id_{f_*A})$ and $\eta_B=\alpha_{f^*B,B}(\id_{f^*B})$.
If $(f^*,f_*)$ is an adjunction, then the functor $f^*$ is left adjoint to $f_*$ and the functor $f_*$ is right adjoint to $f^*$.

\begin{lemma}\label{keylem}
Let $\cA$ and $\cB$ be additive categories. Suppose $(f^*, f_*)$ is an adjunction between functors $f^*\colon\cA\to\cB$ and $f_*\colon\cB\to \cA$. Let $X\in \cA$, $Y\in\cB$.
\begin{enumerate}
\item If $f^*X \neq 0$, then the unit map $\eta_X\colon X \to f_*f^*X$ is non-zero.
\item If $f_*Y \neq 0$, then the counit map $\varepsilon_Y\colon f^*f_*Y \to Y$ is non-zero.
\end{enumerate}
\end{lemma}

\begin{proof}
As the composition $f^* X \mapright{f^*(\eta_X)} f^*f_*f^*X \mapright{\varepsilon_{f^*X}} f^*X$
is the identity on $f^*X$ (see \eqref{eq:unitcounit}), we infer that if $f^*X\neq 0$, then $\eta_X\neq 0$. The proof of (ii) is similar.
\end{proof}

\subsection{}\label{s:transpose}Let $f^*, g^*\colon\cA\to\cB$, $f_*,g_*\colon\cB\to\cA$ be functors and let $(f^*,f_*)$, $(g^*,g_*)$ be adjunctions. Let $\eta$ and $\varepsilon$ denote the unit and counit of the adjunction $(f^*,f_*)$, and let $\eta'$ and $\varepsilon'$ denote the unit and counit of the adjunction $(g^*,g_*)$. Let $\phi\colon f_*\to g_*$ be a natural transformation. The \emph{transpose} $\phi^{\vee}: g^*\to f^*$ is the composition
\begin{equation}\label{eq:transpose} g^*\mapright{\e_{g^*} \eta}g^*f_*f^*\mapright{\e_{g^*}\phi\e_{f^*}}g^*g_*f^*\mapright{\varepsilon'\e_{f^*}}f^*. \end{equation}
The following is a reformulation of \cite[Ch.\ 4 \S7, Thm.\ 2]{MacL}.
\begin{prop}\label{transposeunique}
Suppose $(f^*, f_*)$ and $(g^*, g_*)$ are adjunctions between functors $f^*,g^*\colon\cA\to\cB$ and $f_*,g_*\colon \cB\to\cA$. Let 
\[
\alpha\colon \Hom_{\cA}(f^* - , -)\mapright{\sim}\Hom_{\cB}(-,f_* -), \quad
\alpha'\colon\Hom_{\cA}(g^*-,-)\mapright{\sim}\Hom_{\cB}(-,g_*-),
\]
be the canonical isomorphisms obtained from this data. Let $\phi\colon f_*\to g_*$ be a natural transformation. Then $\phi^{\vee}\colon g^*\to f^*$ is the unique natural transformation such that the following diagram commutes:
\[\xymatrixrowsep{1.5pc}\xymatrixcolsep{1.5pc}\xymatrix{
\Hom_{\cA}(f^* -, -)\ar[r]^-{\circ \phi^{\vee}}\ar[d]_{\alpha}^{\sim}&\Hom_{\cA}(g^* - , -)\ar[d]_{\sim}^{\alpha'} \\
\Hom_{\cB}(-, f_* -)\ar[r]_-{\phi\circ}& \Hom_{\cB}(-, g_* -)
}
\]
\end{prop}
\begin{proof}By definition, $\alpha'^{-1}(\phi\circ\alpha(?))=\varepsilon'\circ \e_{g^*}\phi\circ \e_{g^*f_*}? \circ \e_{g^*}\eta$. Since all morphisms involved are natural transformations, 
\begin{align*}
\varepsilon'\circ \e_{g^*}\phi\circ \e_{g^*f_*}? \circ \e_{g^*}\eta &=\varepsilon'\circ\e_{g^*g_*}?\circ\e_{g^*}\phi\e_{f^*}\circ\e_{g^*}\eta \\
&=?\circ \varepsilon'\e_{f^*}\circ\e_{g^*}\phi\e_{f^*}\circ\e_{g^*}\eta \\
&=?\circ\phi^{\vee}.
\end{align*}
So $\alpha'^{-1}(\phi\circ\alpha(?))=?\circ\phi^{\vee}$ which gives the commutativity of the diagram. As $\alpha$ and $\alpha'$ are isomorphisms, the natural transformation $\circ\phi^{\vee}\colon \Hom_{\cA}(f^*-,-)\to\Hom_{\cA}(g^*-,-)$ is unique. Hence, $\phi^{\vee}$ is unique by the Yoneda Lemma.
\end{proof}

\begin{prop}\label{transposecommute}Suppose $(f^*, f_*)$ and $(g^*, g_*)$ are adjunctions between functors $f^*, g^*\colon\cA\to\cB$ and $f_*,g_*:\cB\to\cA$. Let $\phi\colon f_*\to g_*$ be a natural transformation.
\begin{enumerate}
\item Let $\eta,\varepsilon$ denote the unit and counit of $(f^*,f_*)$ and let $\eta',\varepsilon'$ be the unit and counit of $(g^*, g_*)$. Then the following diagrams commute:
\[\xymatrixcolsep{1.5pc}\xymatrixrowsep{1.5pc}\xymatrix{
f^*f_*\ar[r]^{\varepsilon}&\id \\
g^*f\ar[u]^{\phi^{\vee}\e_{f_*}}\ar[r]_{\e_{g^*}\phi}&g^*g_*\ar[u]_{\varepsilon'}}
\qquad
\xymatrixcolsep{1.5pc}\xymatrixrowsep{1.5pc}\xymatrix{
f_*f^*\ar[r]^{\phi\e_{f^*}}&g_*f^* \\
\id\ar[u]^{\eta}\ar[r]_{\eta'}&g_*g^*\ar[u]_{\e_{g_*}\phi^{\vee}}
}
\]
\item Assume $\cA$ and $\cB$ are additive. Let $\psi\colon f_*\to g_*$ be a natural transformation, then $(\phi+\psi)^{\vee}=\phi^{\vee}+\psi^{\vee}$.
\item Let $(h^*, h_*)$ be an adjunction between functors $h^*\colon\cA\to\cB$ and $h_*\colon\cB\to\cA$. Further, let $\psi\colon g_*\to h_*$ be a natural transformation. Then $(\psi\circ\phi)^{\vee} = \phi^{\vee}\circ \psi^{\vee}$.
\end{enumerate}
\end{prop}

\begin{proof} 
(i) follows from the commutativity of the diagram in Prop.\ \ref{transposeunique}. (ii) follows from our standing assumption that functors between additive categories are additive, i.e., the induced maps on $\Hom$ groups are homomorphisms. (iii) follows from the uniqueness part of Prop.\ \ref{transposeunique}.
\end{proof}

\begin{prop}\label{transposeiso}Let $f^*\colon\cA\to\cB$, $f_*\colon\cB\to\cA$ be functors and let $(f^*,f_*)$ be an adjunction. 
\begin{enumerate}
\item $\e_{f_*}^{\vee}=\e_{f^*}$.
\item Assume $\cA$ and $\cB$ are additive. Then $0^{\vee}=0$.
\item If $e\colon f_*\to f_*$ is idempotent, then $e^{\vee}\colon f^*\to f^*$ is also idempotent.
\end{enumerate}
\end{prop}

\begin{proof}Each of the equalities follows from the uniqueness part of Prop.\ \ref{transposeunique}. Details are left to the reader out of sheer laziness.
\end{proof}

\subsection{}
Let $(f^*,f_*)$ and $(g^*,g_*)$ be adjunctions between functors $g^*\colon\cA\to\cB$, $g_*\colon\cB\to\cA$, $f^*\colon\cB\to\cC$ and $f_*\colon\cC\to\cB$. Then we have the data of four morphisms (units and counits): $\eta\colon \id_{\cB} \to f_*f^*$, $\varepsilon\colon f^*f_* \to \id_{\cC}$, $\eta'\colon \id_{\cA} \to g_*g^*$ and $\varepsilon\colon g^*g_* \to \id_{\cB}$.
It is well known that $f^*g^*$ is left adjoint to $g_*f_*$. It is sometimes useful to have a precise version of this:
let $\overline{\eta}$ and $\overline{\varepsilon}$ be the compositions
\[\id_{\cA}\mapright{\eta'}g_*g^*\mapright{\e_{g_*}\eta\e_{g^*}}g_*f^*f_*g^*
\quad\mbox{and}\quad
f^*g^*g_*f_*\mapright{\e_{f^*}\varepsilon'\e_{f_*}}f^*f_*\mapright{\varepsilon}\id_{\cB},\]
respectively.

\begin{lemma}\label{transposelemma}The natural transformations $\overline{\eta}$ and $\overline{\varepsilon}$ define an adjunction $(f^*g^*, g_*f_*)$. Further, $\varepsilon^{\vee}=\eta'$ and $(\eta')^{\vee}=\varepsilon$.
\end{lemma}

\begin{proof}We have
\begin{align*}
\e_{g_*f_*}\overline{\varepsilon}\circ\overline{\eta}\e_{g_*f_*} &=
\e_{g_*f_*}\varepsilon\circ\e_{g_*f_*f^*}\varepsilon'\e_{f_*}\circ\e_{g_*}\eta\e_{g^*g_*f_*}\circ\eta'\e_{g_*f_*} \\
&= \e_{g_*f_*}\varepsilon\circ\e_{g_*}\eta\e_{f_*}\circ\e_{g_*}\varepsilon'\e_{f_*}\circ\eta'\e_{g_*f_*} \\
&=\e_{g_*f_*},
\end{align*}
where the first equality is the definition of $\overline{\varepsilon}$ and $\overline{\eta}$, the second equality holds due to $\eta$ and $\varepsilon'$ being natural transformations and the last equality follows from the definition of unit/counit \eqref{eq:unitcounit}.
The proof that $\overline{\varepsilon}\e_{f^*g^*}\circ\e_{f^*g^*}\overline{\eta}=\e_{f^*g^*}$ is similar. Thus, $\overline{\eta}$ and $\overline{\varepsilon}$ define an adjunction $(f^*g^*, g_*f_*)$.
Further,
\[\varepsilon^{\vee}=\varepsilon\e_{f^*f_*}\circ\overline{\eta} 
=\varepsilon\e_{f^*f_*}\circ \e_{f^*}\eta \e_{f_*} \circ \eta' 
=\eta',\]
where the first equality is the definition of transpose \eqref{eq:transpose}, the second equality is the definition of $\overline{\eta}$ and the last equality follows from the definition of the unit/counit \eqref{eq:unitcounit}.
Similarly,
\[(\eta')^{\vee}=\varepsilon\circ \e_{f^*}\varepsilon'\e_{f_*}\circ \e_{f^*f_*}\eta'=\varepsilon.\qedhere\]
\end{proof}

\subsection{}Let $(h^*, h_*)$ be another adjunction, between functors $h^*\colon \cZ\to\cA$, $h_*\colon \cA\to\cZ$. Using the procedure above there are, \emph{a priori}, two different ways to define an adjunction $(f^*g^*h^*, h_*g_*f_*)$: either first construct an adjunction $(g^*h^*, h_*g_*)$ and then an adjunction $(f^*(g^*h^*), (h_*g_*)f_*)$ or first construct an adjunction $(f^*g^*, g_*f_*)$ and then an adjunction $((f^*g^*)h^*, h_*(g_*f_*))$. Let
\[ \Hom_{\cC}(f^*g^*h^*X, Y)\mapright{\alpha} \Hom_{\cB}(g^*h^*X, f_*Y) \mapright{\alpha'}\Hom_{\cZ}(X, h_*g_*f_* Y),\]
\[ \Hom_{\cC}(f^*g^*h^*X, Y)\mapright{\alpha''}\Hom_{\cA}(h^*X, g_*f_*Y) \mapright{\alpha'''}\Hom_{\cZ}(X, h_*g_*f_* Y),\]
$X\in Z$, $Y\in\cC$, be the sequences of canonical isomorphisms obtained this way.
\begin{prop}The following diagram commutes.
\[\xymatrix{
\Hom_{\cC}(f^*g^*h^*X, Y) \ar[r]^{\alpha''}\ar[d]_{\alpha}& \Hom_{\cA}(h^*X, g_*f_*Y) \ar[d]^{\alpha'''} \\
\Hom_{\cB}(g^*h^*X, f_*Y) \ar[r]^{\alpha'} & \Hom_{\cZ}(X, h_*g_*f_*Y)
}\]
\end{prop}
\begin{proof}
Both $\alpha'\circ \alpha$ and $\alpha'''\circ \alpha''$ are equal to the composite canonical isomorphism
\begin{align*}
&\Hom_{\cC}(f^*g^*h^*X, Y)\mapright{\sim}\Hom_{\cB}(g^*h^*X, f_*Y)\mapright{\sim} \Hom_{\cA}(h^*X, g_*f_*Y)\\
&\mapright{\sim} \Hom_{\cZ}(X, h_*g_*f_*Y).\qedhere
\end{align*}
\end{proof}

\subsection{}
Let $f_!,g_!\colon\cA\to\cB$, $f^!,g^!\colon\cB\to\cA$ be functors and let $(f_!,f^!)$, $(g_!, g^!)$ be adjunctions. Write $\eta$ and $\varepsilon$ for the unit and counit of $(f_!, f^!)$, and write $\eta'$ and $\varepsilon'$ for the unit and counit of $(g_!, g^!)$. Suppose $\psi\colon g_!\to f_!$ is a natural transformation. Then the \emph{right transpose} 
$\vphantom{\psi}^{\vee}\psi\colon f^!\to g^!$ is the composition
\begin{equation}\label{eq:righttranspose} 
f^!\mapright{\eta'\e_{f^!}}g^!g_!f^!\mapright{\e_{g^!}\psi'\e_{f^!}}g^!f_!f^!\mapright{\e_{g^!}\varepsilon}g^!.
\end{equation}
The next result allows us to transport all the statements for transposes to right transposes.
\begin{prop}Let $(f_!, f^!)$ and $(g_!, g^!)$ be adjunctions between functors $f_!,g_!\colon\cA\to\cB$ and $f^!,g^!\colon\cB\to\cA$. Let $\phi\colon f^!\to g^!$ be a natural transformation. Then $\vphantom{\phi}^{\vee}(\phi^{\vee})=\phi$. Similarly, if $\psi\colon g_! \to f_!$ is a natural transformation, then $(\vphantom{\psi}^{\vee}\psi)^{\vee} = \psi$
\end{prop}
\begin{proof}Let $\eta,\varepsilon$ be the unit and counit of $(f_!, f^!)$ and let $\eta',\varepsilon'$ be the unit and counit of $(g_!,g^!)$. Then
\begin{align*}
\vphantom{\phi}^{\vee}(\phi^\vee)&=\e_{g^!}\varepsilon\circ \e_{g^!}\varepsilon'\e_{f_!f^!}\circ\e_{g^!g_!}\phi\e_{f_!f^!}\circ \e_{g^!g_!}\eta\e_{f^!}\circ \eta'\e_{f^!} \\
&=\e_{g^!}\varepsilon\circ \e_{g^!}\varepsilon'\e_{f_!f^!}\circ\e_{g^!g_!}\phi\e_{f_!f^!}\circ \eta'\e_{f^!f_!f^!}\circ \eta\e_{f^!} \\
&=\e_{g^!}\varepsilon\circ \e_{g^!}\varepsilon'\e_{f_!f^!}\circ \eta'\e_{g^!f_!f^!}\circ\phi\e_{f_!f^!}\circ\eta\e_{f^!} \\
&= \e_{g^!}\varepsilon\circ \phi \e_{f_!f^!}\circ\eta\e_{f^!} \\
&= \phi \circ \e_{f^!}\varepsilon\circ \eta\e_{f^!} \\
&= \phi.
\end{align*}
The first equality is by the definition of transpose \eqref{eq:transpose} and right transpose \eqref{eq:righttranspose}, the second, third and fifth equalities are due to the fact that all morphisms involved are natural transformations. The fourth and last equalities follow from the definition of the unit/counit \eqref{eq:unitcounit}.
The proof that $(\vphantom{\psi}^{\vee}\psi)^{\vee}=\psi$ is similar.
\end{proof}

\section{Complexes of functors}\label{s:complexesoffunctors}
\subsection{}Let $\cA,\cB$ be additive categories. Write $\HOM(\cA,\cB)$ for the additive category of functors $\cA\to \cB$ with morphisms given by natural transformations. Let $\cC$ be another additive category. Let $F\in\Comp(\HOM(\cB,\cC))$, $G\in\Comp(\HOM(\cA,\cB))$. Define the object $F G$ in $\Comp(\HOM(\cA,\cC))$ to be the complex whose degree $n$ component is $\bigoplus_{i+j=n}F^iG^j$ with differential
\[ d_{FG}\colon F^iG^j \to F^{i+1}G^j \oplus F^iG^{j+1}, \qquad d_{FG}=d_F\e_{G^j}+(-1)^i\e_{F^i}d_G. \]
\begin{remark}
$FG$ is the total complex of the double complex $\{F^iG^j\}_{i,j}$.
\end{remark}
\begin{prop}\label{2catassoc}
Let $\cA,\cB,\cC,\cD$ be additive categories. Let $F\in\Comp(\HOM(\cC,\cD))$, $G\in\Comp(\HOM(\cB,\cC))$, $H\in\Comp(\HOM(\cA,\cB))$. Then 
$(FG)H = F(GH)$.
\end{prop}

\begin{proof}
The degree $n$ component of both $(FG)H$ and $F(GH)$ is $\bigoplus_{i+j+k=n}F^iG^jH^k$. It remains to check that the differentials on both sides coincide. 
The differential for $(FG)H$, 
$d_{(FG)H}\colon F^iG^jH^k \to F^{i+1}G^jH^k \oplus F^iG^{j+1}H^k \oplus F^iG^jH^{k+1}$
is
\begin{align*}
d_{(FG)H} &= d_{FG}\e_{H^k}+(-1)^{i+j}\e_{F^iG^j}d_H \\
&=d_F\e_{G^jH^k}+(-1)^i\e_{F^i}d_G\e_{H^k} + (-1)^{i+j}\e_{F^iG^j}d_H.
\end{align*}
The differential for $F(GH)$,
$d_{F(GH)}\colon F^iG^jH^k \to F^{i+1}G^jH^k \oplus F^iG^{j+1}H^k \oplus F^iG^jH^{k+1}$
is
\begin{align*}
d_{F(GH)}&=d_{F}\e_{G^jH^k}+(-1)^i\e_{F^i}d_{GH} \\
&= d_{F}\e_{G^jH^k} + (-1)^i \e_{F^i}d_G\e_{H^k} + (-1)^{i+j}\e_{F^iG^j}d_H.\qedhere
\end{align*}
\end{proof}

\subsection{}Let $\cA$ and $\cB$ be additive categories.
Let $(f^*_i, f_{i*})$, $i\in\ZZ$, be adjunctions between functors $f_{i*}\colon\cA\to\cB$ and $f_{i}^*\colon\cB\to\cA$. Suppose we have a complex of functors
\[ F_* = \cdots\mapright{d_{-2}}f_{-1*}\mapright{d_{-1}} f_{0*} \mapright{d_0} f_{1*} \mapright{d_{1}} \cdots, \]
with $f_{0*}$ in degree $0$. Set
\[ F^* = \cdots \mapright{d_1^{\vee}}f_1^*\mapright{d_{0}^{\vee}} f_0^* \mapright{d_{-1}^{\vee}} f_{-1}^*\mapright{d_{-2}^{\vee}}\cdots, \]
with $f^*_0$ in degree $0$. Then Prop.\ \ref{transposecommute} (iii) and Prop.\ \ref{transposeiso} (ii) imply that $F^*$ is also a complex.
The degree $0$ term of $F^*F$ is $\bigoplus_{i\in\ZZ} f_i^*f_{i*}$. 
View the identity functor as a complex concentrated in degree $0$. Define $\ev\colon F^*F_*\to\id$ by
\[ \left(\begin{smallmatrix}\cdots & -\varepsilon_{-2}&-\varepsilon_{-1}&\varepsilon_0&\varepsilon_1&-\varepsilon_2&-\varepsilon_3&\varepsilon_4&\varepsilon_5&\cdots\end{smallmatrix}\right)\colon\bigoplus_{i\in\ZZ} f_i^*f_{i*}\to \id,\]
where $\varepsilon_i$ is the counit of the adjunction $(f_i^*, f_{i*})$.
The differential on the degree $-1$ term of $F^*F_*$ is given by
\[\left(\begin{smallmatrix}
d_i^{\vee}\e_{f_{i*}} \\
(-1)^{i+1}\e_{f^*_{i+1}}d_i
\end{smallmatrix}\right)
\colon f^*_{i+1}f_{i*}\to f^*_if_{i*} \oplus f^*_{i+1}f_{i+1*}.
\]
This combined with Prop.\ \ref{transposecommute} (i) implies that $\ev$ is a chain map.
Similarly, the degree $0$ term of $F_*F^*$ is $\bigoplus_{i\in\ZZ}f_{i*}f_i^*$. Define
$\coev\colon\id\to F_*F^*$ by
\[ \left(\begin{smallmatrix}
\vdots \\
-\eta_{-2} \\
-\eta_{-1} \\
\eta_0 \\
\eta_1 \\
-\eta_2 \\
-\eta_3\\
\eta_4\\ 
\eta_5 \\
\vdots
\end{smallmatrix}\right)\colon
\id\to \bigoplus_{i\in\ZZ}f_{i*}f_i^*,.\]
where $\eta_i$ is the counit of the adjunction $(f^*_i, f_{i*})$.
The differential on the degree $0$ term is given by
\[ \left(\begin{smallmatrix}
d_i\e_{f_i^*}\\
(-1)^i\e_{f_{i*}}d_{i-1}^{\vee}
\end{smallmatrix}\right)\colon f_{i*}f_i^*\to f_{i+1*}f_i^*\oplus f_{i*}f_{i-1}^*. \]
This combined with Prop.\ \ref{transposecommute} (i) gives that $\coev$ is a chain map.
\begin{prop}\label{complexesadjoint}The compositions
\[\xymatrix{F_* \ar[r]^-{\coev\e_{F_*}}& F_*F^*F_*\ar[r]^-{\e_{F_*}\ev}&F_* }\quad \mbox{and}\quad\xymatrix{F^*\ar[r]^-{\e_{F^*}\coev}&F^*F_*F^*\ar[r]^-{\ev\e_{F^*}}&F^*} \]
are equal to the identity on $F_*$ and $F^*$, respectively.
\end{prop}

\begin{proof}This follows from the corresponding properties of $\eta_i$ and $\varepsilon_i$ (cf.\ example \ref{prooffromscratch}).
\end{proof}

\section{A general construction}\label{s:genconst}
\subsection{}\label{s:trianggen}Let $\cT$ be a triangulated category. Let $\cA,\cB\subseteq \cT$ be subcategories. For $X\in \cT$ write $[X]\in\cA$ (resp.\ $\cB$) if there exists an object in $\cA$ (resp.\ $\cB$) isomorphic to $X$. Define
\begin{align*}
\cA * \cB = \{Y\in\cT\,|\,&\mbox{there is a distinguished triangle $X\to Y\to Z\leadsto$}\\
&\mbox{with $[X]\in\cA$ and $Z\in\cB$}\}.
\end{align*}
The operation $*$ is associative (see \cite[Lemma 1.3.10]{BBD}). Inductively define $\cA^{*i}$, $i\in\ZZ_{\geq 0}$, by $\cA^{*0}=0$ and $\cA^{*i+1} = \cA *\cA^{*i}$. Set $\cA^{*\infty}=\bigcup_{i\in\ZZ_{\geq 0}} \cA^{*i}$.
It is evident that $X\in\cA^{*n}$ if and only if $X$ is filtered by some $Y_1,\ldots, Y_n\in\cA$.

\begin{lemma}\label{devissagetriang}Let $\cT$ and $\cT'$ be triangulated categories. Let $\cL\subset \cT$ be a subcategory (not necessarily triangulated). Suppose that $\cL^{*\infty}=\cT$. Let $f,g\colon \cT\to\cT'$ be exact functors and let $\varepsilon\colon f\to g$ be a morphism of exact functors. If $\varepsilon_L\colon fL \to gL$ is an isomorphism for each $L\in\cL$, then $\varepsilon\colon f\to g$ is an isomorphism.
\end{lemma}

\begin{proof}Proceed by induction, assume that if $i<n$, then $\varepsilon_L\colon fL\to gL$ is an isomorphism for each $L\in \cL^{*i}$. Let $M\in\cL^{*n}$, then we have a distinguished triangle $N \to M \to L\leadsto$ with $N\in\cL^{*n-1}$ and $L\in\cL$. So we obtain a commutative diagram
\[ \xymatrix{
N\ar[r]\ar[d]_{\varepsilon_N}^{\sim}& M\ar[r]\ar[d]_{\varepsilon_M} & L\ar@{~>}[r]\ar[d]_{\varepsilon_L}^{\sim}& \\
N\ar[r]& M\ar[r] & L\ar@{~>}[r]&
}\]
The outer vertical arrows are isomorphisms by hypothesis. This forces the middle arrow to also be an isomorphism.
\end{proof}

\subsection{}Let $\cA$ and $\cB$ be abelian categories. Let $F\in\Comp(\HOM(\cA,\cB))$. Assume that each component of $F$ is an exact functor. For further simplicity assume that $F$ is bounded. Then $F$ defines a functor $\Comp(\cA) \to\Comp(\cB)$ (it is defined exactly as the `composition' in \S\ref{s:complexesoffunctors}). Since each component of $F$ is exact, this gives an exact functor $\Db(\cA)\to\Db(\cB)$.

The following is originally due to Rickard \cite[Thm.\ 2.1]{Ri} (also see \cite[\S2.2.3]{Ro}, \cite[Lemma 4.1.1]{ABG}, \cite[Thm.\ 7.3.16]{Vo}.
\begin{thm}\label{mainthm}
Let $\cA$ and $\cB$ be abelian categories.
Assume each object in $\cA$ has finite length.
Let $(\pi^*,\pi_*)$ and $(\pi_*, \pi^!)$ be adjunctions between exact functors $\pi_*\colon\cA\to\cB$ and $\pi^*,\pi^!\colon\cB\to\cA$. Then we have the data of four morphisms (units and counits):
\[ \eta\colon \id_{\cB} \to \pi_*\pi^*, \quad \varepsilon\colon \pi^*\pi_*\to \id_{\cA}, \quad \eta'\colon\id_{\cA} \to \pi^!\pi_*, \quad \varepsilon'\colon \pi_*\pi^! \to\id_{\cB}. \]
Define complexes of functors $\Theta^*$ and $\Theta^!$:
\[ \Theta^*=0\to \pi^*\pi_* \mapright{\varepsilon} \id_{\cA} \to 0 \qquad\mbox{and}\qquad \Theta^!= 0 \to \id_{\cA} \mapright{\eta'} \pi^!\pi_* \to 0, \]
with $\pi^*\pi_*$ and $\pi^!\pi_*$ in degree $0$. By Lemma \ref{transposelemma} and Prop.\ \ref{complexesadjoint}, $\Theta^*$ is left adjoint to $\Theta^!$. Fix an adjunction $(\Theta^*,\Theta^!)$ and denote the unit by $\coev$ and the counit by $\ev$.
\begin{enumerate}
\item If $[\pi^*\pi_*\pi^!\pi_* X] = [\pi^*\pi_* X] + [\pi^!\pi_*X]$ in $K_0(\cA)$ for each $X\in\cA$, then $\ev\colon\Theta^*\Theta^!\to\id$ is an isomorphism of functors on $\Db(\cA)$.
\item If $[\pi^!\pi_*\pi^*\pi_* X] = [\pi^*\pi_* X] + [\pi^!\pi_*X]$ in $K_0(\cA)$ for each $X\in\cA$, then $\coev\colon\id\to\Theta^!\Theta^*$ is an isomorphism of functors on $\Db(\cA)$.
\end{enumerate}
\end{thm}

\begin{proof}
By definition, the functor $\Theta^*\Theta^!$ is given by the complex
\[\xymatrixcolsep{3.5pc}\xymatrix{
0\ar[r]&\pi^*\pi_*\ar[r]^-{\left(\begin{smallmatrix}\varepsilon\\\e_{\pi^*\pi_*}\eta'\end{smallmatrix}\right)}&\id_{\cA}\oplus\pi^*\pi_*\pi^!\pi_*\ar[r]^-{\left(\begin{smallmatrix}-\eta'&\varepsilon\e_{\pi^!\pi_*}\end{smallmatrix}\right)}&\pi^!\pi_*\ar[r]& 0.}\]
By definition of the unit $\eta'$ and the counit $\varepsilon'$, the composition 
\[\xymatrixcolsep{3.5pc}\xymatrix{\pi^*\pi_* \ar[r]^-{\e_{\pi^*\pi_*}\eta'} & \pi^*\pi_*\pi^!\pi_* \ar[r]^-{\e_{\pi^*}\varepsilon'\e_{\pi_*}} &\pi^*\pi_*}\]
is the identity on $\pi^*\pi_*$. Thus,
$\xymatrixcolsep{3.5pc}\xymatrix{
\pi^*\pi_*\ar[r]^-{\left(\begin{smallmatrix}\varepsilon\\\e_{\pi^*\pi_*}\eta'\end{smallmatrix}\right)}&\id_{\cA}\oplus\pi^*\pi_*\pi^!\pi_*
}$
is a monomorphism. A similar argument shows that
$\xymatrixcolsep{3.5pc}\xymatrix{
\id_{\cA}\oplus\pi^*\pi_*\pi^!\pi_*\ar[r]^-{\left(\begin{smallmatrix}-\eta'&\varepsilon\e_{\pi^!\pi_*}\end{smallmatrix}\right)}&\pi^!\pi_*
}$
is an epimorphism. Hence, if $X\in\cA$, then $\Theta^*\Theta^!X$ is isomorphic (in $\Db(\cA)$) to an object in $\cA$. Let $L\in \cA$ be simple, then under the hypothesis of (i):
\[[\Theta^*\Theta^! L] = [\pi^*\pi_*\pi^!\pi_*L]+[L] - [\pi^*\pi_*L]-[\pi^!\pi_*L]
= [L] \quad \mbox{in $K_0(\cA)$}.\]
This forces $\Theta^*\Theta^!L\simeq L$. Lemma \ref{keylem} (ii) gives that $\ev\colon\Theta^*\Theta^!L\to L$ is non-zero. Since $L$ is simple, this implies that $\ev\colon\Theta^*\Theta^!L\to L$ is an isomorphism. As every object in $\cA$ is of finite length, every object in $\cA$ is filtered by simple objects. Further, every object in $\Db(\cA)$ is filtered by shifts of objects in $\cA$. Thus, every object in $\Db(\cA)$ is filtered by shifts of the simple objects in $\cA$. Applying Lemma \ref{devissagetriang} now gives (i).

The functor $\Theta^!\Theta^*$ is given by the complex
\[ \xymatrixcolsep{3.5pc}\xymatrix{0\ar[r]& \pi^*\pi_*\ar[r]^-{\left(\begin{smallmatrix}\eta'\e_{\pi^*\pi_*} \\ -\varepsilon\end{smallmatrix}\right)}&\pi^!\pi_*\pi^*\pi_*\oplus\id_{\cA}\ar[r]^-{\left(\begin{smallmatrix}\e_{\pi^!\pi_*}\varepsilon & \eta'\end{smallmatrix}\right)} &\pi^!\pi_*\ar[r]&0}. \]
Now an argument similar to the one for (i) gives (ii).
\end{proof}

\begin{example}\label{prooffromscratch}We will now work out a `proof from scratch' of Thm.\ \ref{mainthm} in the special case $\pi^!=\pi^*$.

Let $\cA$ and $\cB$ be abelian categories. Assume each object in $\cA$ has finite length. Let $(\pi^*,\pi_*)$ and $(\pi_*,\pi^*)$ be adjunctions between exact functors $\pi_*\colon\cA\to\cB$ and $\pi^*\colon\cB\to\cA$. Then we have the data of four morphisms (units and counits):
\[ \eta\colon\id_{\cB}\to\pi_*\pi^*, \quad \varepsilon\colon\pi^*\pi_*\to\id_{\cA}, \quad \eta'\colon\id_{\cA}\to\pi^*\pi_*\quad \varepsilon'\colon\pi_*\pi^*\to\id_{\cB}.\]
Let $\Theta^*=0\to\pi^*\pi_*\mapright{\varepsilon}\id_{\cA}\to 0$ and $\Theta^!=0\to\id_{\cA}\mapright{\eta'}\pi^*\pi_*\to 0$ with $\pi^*\pi_*$ in degree $0$ in both cases.
Let's show that $\Theta^*$ is left adjoint to $\Theta^!$.
It is helpful to keep track of terms in this computation `in color' (I apologize to the reader trying to read this in monochrome).
The functor $\color{blue}\Theta^*\color{red}\Theta^!$ is given by the complex
\[\xymatrixcolsep{3.5pc}\xymatrix{
\color{blue}\pi^*\pi_*\color{red}\id_{\cA}\ar[r]^-{\left(\begin{smallmatrix}\varepsilon\\\e_{\color{blue}\pi^*\pi_*}\eta'\end{smallmatrix}\right)}&\color{blue}\id_{\cA}\color{red}\id_{\cA}\color{black}\oplus\color{blue}\pi^*\pi_*\color{red}\pi^*\pi_*\ar[r]^-{\left(\begin{smallmatrix}-\eta'&\varepsilon\e_{\color{red}\pi^*\pi_*}\end{smallmatrix}\right)}&\color{blue}\id_{\cA}\color{red}\pi^*\pi_*}\]
with $\color{blue}\id_{\cA}\color{red}\id_{\cA}\color{black}\oplus\color{blue}\pi^*\pi_*\color{red}\pi^*\pi_*$ in degree $0$.
Define $\ev\colon\color{blue}\Theta^*\color{red}\Theta^!\color{black}\to\id_{\cA}$ by
\[\xymatrixcolsep{3.5pc}\xymatrix{
\color{blue}\pi^*\pi_*\color{red}\id_{\cA}\ar[d]\ar[r]^-{\left(\begin{smallmatrix}\varepsilon\\\e_{\color{blue}\pi^*\pi_*}\eta'\end{smallmatrix}\right)}&\color{blue}\id_{\cA}\color{red}\id_{\cA}\color{black}\oplus\color{blue}\pi^*\pi_*\color{red}\pi^*\pi_*
\ar[d]_{\left(\begin{smallmatrix}
-\id &
\varepsilon\circ\color{blue}\e_{\pi^*}\color{black}\varepsilon'\color{red}\e_{\pi_*}
\end{smallmatrix}\right)}
\ar[r]^-{\left(\begin{smallmatrix}-\eta'&\varepsilon\e_{\color{red}\pi^*\pi_*}\end{smallmatrix}\right)}&\color{blue}\id_{\cA}\color{red}\pi^*\pi_*\ar[d]\\
0\ar[r]&\id_{\cA}\ar[r]&0
}\]
We have
\[
\left(\begin{smallmatrix}
-\id &
\varepsilon\circ\e_{\pi^*}\varepsilon'\e_{\pi_*}
\end{smallmatrix}\right)
\circ
\left(\begin{smallmatrix}\varepsilon\\\e_{\pi^*\pi_*}\eta'\end{smallmatrix}\right)
= -\varepsilon + \varepsilon\circ\e_{\pi^*}\varepsilon'\e_{\pi_*}\circ\e_{\pi^*\pi_*}\eta'=0,
\]
where the last equality is by the definition of the unit $\eta'$ and the counit $\varepsilon'$. Thus, $\ev$ is a chain map. The functor $\color{green}\Theta^!\color{blue}\Theta^*$ is given by the complex
\[ \xymatrixcolsep{3.5pc}\xymatrix{\color{green}\id_{\cA}\color{blue}\pi^*\pi_*\ar[r]^-{\left(\begin{smallmatrix}\eta'\e_{\color{blue}\pi^*\pi_*} \\ -\varepsilon\end{smallmatrix}\right)}&\color{green}\pi^*\pi_*\color{blue}\pi^*\pi_*\color{black}\oplus\color{green}\id_{\cA}\color{blue}\id_{\cA}\ar[r]^-{\left(\begin{smallmatrix}\e_{\color{green}\pi^*\pi_*}\varepsilon & \eta'\end{smallmatrix}\right)} &\color{green}\pi^*\pi_*\color{blue}\id_{\cA}} \]
with $\color{green}\pi^*\pi_*\color{blue}\pi^*\pi_*\color{black}\oplus\color{green}\id_{\cA}\color{blue}\id_{\cA}$ in degree $0$.
Define $\coev\colon\id_{\cA}\to\color{green}\Theta^!\color{blue}\Theta^*$ by
\[ \xymatrixrowsep{4.5pc}\xymatrixcolsep{3.5pc}\xymatrix{
0\ar[r]\ar[d]&\id_{\cA}
\ar[d]_-{\left(\begin{smallmatrix}
\e_{\pi^*}\eta\e_{\pi_*}\circ\eta' \\
-\id
\end{smallmatrix}\right)}
\ar[r]&0\ar[d]\\
\color{green}\id_{\cA}\color{blue}\pi^*\pi_*\ar[r]^-{\left(\begin{smallmatrix}\eta'\e_{\color{blue}\pi^*\pi_*} \\ -\varepsilon\end{smallmatrix}\right)}&\color{green}\pi^*\pi_*\color{blue}\pi^*\pi_*\color{black}\oplus\color{green}\id_{\cA}\color{blue}\id_{\cA}\ar[r]^-{\left(\begin{smallmatrix}\e_{\color{green}\pi^*\pi_*}\varepsilon & \eta'\end{smallmatrix}\right)} &\color{green}\pi^*\pi_*\color{blue}\id_{\cA}
} \]
We have
\[\left(\begin{smallmatrix}\e_{\pi^*\pi_*}\varepsilon & \eta'\end{smallmatrix}\right)
\circ
\left(\begin{smallmatrix}
\e_{\pi_*}\eta\e_{\pi^*}\circ\eta' \\
-\id
\end{smallmatrix}\right)
=
\e_{\pi^*\pi_*}\varepsilon\circ\e_{\pi^*}\eta\e_{\pi_*}\circ\eta' - \eta' = 0,
\]
where the last equality is by the definition of the unit $\eta$ and the counit $\varepsilon$. Thus, $\coev$ is also a chain map.
The functor $\color{green}\Theta^!\color{blue}\Theta^*\color{red}\Theta^!$ is given by the complex (we omit the differential since it is no longer relevant to the discussion)
\begin{align*} 0\to &\color{green}\id_{\cA}\color{blue}\pi^*\pi_*\color{red}\id_{\cA}\color{black}\to\color{green}\id\color{blue}\pi^*\pi_*\color{red}\pi^*\pi_*\color{black}\oplus\color{green}\pi^*\pi_*\color{blue}\pi^*\pi_*\color{red}\id_{\cA}\color{black}\oplus\color{green}\id_{\cA}\color{blue}\id_{\cA}\color{red}\id_{\cA}\color{black} \to \\ 
&\to \color{green}\pi^*\pi_*\color{blue}\pi^*\pi_*\color{red}\pi^*\pi_*\color{black}\oplus\color{green}\id_{\cA}\color{blue}\id_{\cA}\color{red}\pi^*\pi_*\color{black}\oplus\color{green}\pi^*\pi_*\color{blue}\id_{\cA}\color{red}\id_{\cA}\color{black}\to \color{green}\pi^*\pi_*\color{blue}\id_{\cA}\color{red}\pi^*\pi_*\color{black}\to 0 \end{align*}
with $\color{green}\pi^*\pi_*\color{blue}\pi^*\pi_*\color{red}\pi^*\pi_*\color{black}\oplus\color{green}\id_{\cA}\color{blue}\id_{\cA}\color{red}\pi^*\pi_*\color{black}\oplus\color{green}\pi^*\pi_*\color{blue}\id_{\cA}\color{red}\id_{\cA}$ in degree $0$. The composition $\xymatrix{\color{red}\Theta^!\ar[r]^-{\coev\e_{\color{red}\Theta^!}}&\color{green}\Theta^!\color{blue}\Theta^*\color{red}\Theta^!\ar[r]^-{\e_{\color{green}\Theta^!}\ev}&\color{green}\Theta^!}$ is given by
\[\xymatrix{
\color{red}\id_{\cA}\ar[r]\ar[d]_{
\left(\begin{smallmatrix}
0 \\ \e_{\color{green}\pi^*}\color{black}\eta\e_{\color{blue}\pi_*}\circ\eta' \\
-\id
\end{smallmatrix}\right)
}
&\color{red}\pi^*\pi_*\ar[d]_{
\left(\begin{smallmatrix}
\e_{\color{green}\pi^*}\eta\e_{\color{blue}\pi_*\color{red}\pi^*\pi_*}\circ\eta'\e_{\color{red}\pi^*\pi_*} \\
-\e_{\color{red}\pi^*\pi_*} \\
0
\end{smallmatrix}\right)
}\\
\color{green}\id_{\cA}\color{blue}\pi^*\pi_*\color{red}\pi^*\pi_*\color{black}\oplus\color{green}\pi^*\pi_*\color{blue}\pi^*\pi_*\color{red}\id_{\cA}\color{black}\oplus\color{green}\id_{\cA}\color{blue}\id_{\cA}\color{red}\id_{\cA}\color{black}\ar[r]\ar[d]_{
\left(\begin{smallmatrix}
\varepsilon\circ\e_{\color{blue}\pi^*}\varepsilon'\e_{\color{red}\pi^*}
&0
&-\id
\end{smallmatrix}\right)}
&\color{green}\pi^*\pi_*\color{blue}\pi^*\pi_*\color{red}\pi^*\pi_*\color{black}\oplus\color{green}\id_{\cA}\color{blue}\id_{\cA}\color{red}\pi^*\pi_*\color{black}\oplus\color{green}\pi^*\pi_*\color{blue}\id_{\cA}\color{red}\id_{\cA}\color{black}\ar[d]_{
\left(\begin{smallmatrix}
\e_{\color{green}\pi^*\pi_*}\varepsilon\circ\e_{\color{green}\pi^*\pi_*\color{blue}\pi^*}\varepsilon'\e_{\color{red}\pi_*}
&0
&-\e_{\color{green}\pi^*\pi_*}
\end{smallmatrix}\right)}\\
\color{green}\id_{\cA}\ar[r]&\color{green}\pi^*\pi_*
}\]
It is evident that the vertical composition on the left is the identity. Furthermore,
\begin{align*}
&\left(\begin{smallmatrix}\e_{\pi^*\pi_*}\varepsilon\circ\e_{\pi^*\pi_*\pi^*}\varepsilon'\e_{\pi_*} & 0 & -\e_{\pi^*\pi_*}\end{smallmatrix}\right)\circ
\left(\begin{smallmatrix}\e_{\pi^*}\eta\e_{\pi_*\pi^*\pi_*}\circ\eta'\e_{\pi^*\pi_*} \\
-\e_{\pi^*\pi_*} \\ 0
\end{smallmatrix}\right) \\ 
&=\e_{\pi^*\pi_*}\varepsilon\circ\e_{\pi^*\pi_*\pi^*}\varepsilon'\e_{\pi_*} \circ \e_{\pi^*}\eta\e_{\pi_*\pi^*\pi_*}\circ\eta'\e_{\pi^*\pi_*} \\
&=\e_{\pi^*\pi_*}\varepsilon\circ\e_{\pi^*}\eta\e_{\pi^*}\circ\e_{\pi^*}\varepsilon'\e_{\pi^*}\circ \eta'\e_{\pi^*\pi_*} \\
&= \e_{\pi^*\pi_*}.
\end{align*}
The second equality is due to $\eta$ and $\varepsilon'$ being natural transformations. The third equality is by the definition of the units $\eta,\eta'$ and the counits $\varepsilon,\varepsilon'$.
So the vertical composition on the right is also the identity. Thus, the composition
$\xymatrix{\Theta^!\ar[r]^-{\coev\e_{\Theta^!}}&\Theta^!\Theta^*\Theta^!\ar[r]^-{\e_{\Theta^!}\ev}&\Theta^!}$ is the identity on $\Theta^!$.
A similar computation shows that the composition $\xymatrix{\Theta^*\ar[r]^-{\e_{\Theta^*}\coev}&\Theta^*\Theta^!\Theta^*\ar[r]^-{\ev\e_{\Theta^*}}&\Theta^*}$ is the identity on $\Theta^*$. Hence, $\Theta^*$ is left adjoint to $\Theta^!$.

Now let $L\in\cA$ be simple. Then $\Theta^*\Theta^!L$ is the complex
\[\xymatrixcolsep{3.5pc}\xymatrix{
0\ar[r]&\pi^*\pi_*L\ar[r]^-{\left(\begin{smallmatrix}\varepsilon_L\\ \pi^*\pi_*(\eta'_L)\end{smallmatrix}\right)}&L\oplus\pi^*\pi_*\pi^*\pi_*L\ar[r]^-{\left(\begin{smallmatrix}-\eta'_L&\varepsilon_{\pi^*\pi_*L}\end{smallmatrix}\right)}&\pi^*\pi_*L\ar[r]& 0}\]
with $L\oplus\pi^*\pi_*\pi^*\pi_*L$ in degree $0$.
By definition of the unit $\eta'$ and the counit $\varepsilon'$, the composition 
\[\xymatrixcolsep{3.5pc}\xymatrix{\pi^*\pi_*L \ar[r]^-{\pi^*\pi_*(\eta'_L)} & \pi^*\pi_*\pi^*\pi_*L \ar[r]^-{\pi^*(\varepsilon'_{\pi_*L})} &\pi^*\pi_*L}\]
is the identity on $\pi^*\pi_*L$. Thus,
$\xymatrixcolsep{3.5pc}\xymatrix{
\pi^*\pi_*L\ar[r]^-{\left(\begin{smallmatrix}\varepsilon_L\\ \pi^*\pi_*(\eta'_L)\end{smallmatrix}\right)}&L\oplus\pi^*\pi_*\pi^*\pi_*L
}$
is a monomorphism. Similarly,
$\xymatrixcolsep{4pc}\xymatrix{
L\oplus\pi^*\pi_*\pi^*\pi_*L\ar[r]^-{\left(\begin{smallmatrix}-\eta'_L&\varepsilon_{\pi^*\pi_*L}\end{smallmatrix}\right)}&\pi^*\pi_*L
}$
is an epimorphism. Thus, $\Theta^*\Theta^!$ is isomorphic (in $\Db(\cA)$) to its zeroth cohomology $H^0(\Theta^*\Theta^!L)$.
Assume that $[\pi^*\pi_*\pi^*\pi_*L]=2[\pi^*\pi_*L]$ in $K_0(\cA)$ for each simple $L\in\cA$, then
\[[H^0(\Theta^*\Theta^!L)]=[\Theta^*\Theta^! L] = [\pi^*\pi_*\pi^*\pi_*L]+[L] - 2[\pi^*\pi_*L]
= [L].\]
This forces $H^0(\Theta^*\Theta^!L)$ and hence $\Theta^*\Theta^!L$ to be isomorphic to $L$. Lemma \ref{keylem} (ii) gives that $\ev\colon\Theta^*\Theta^!L\to L$ is non-zero. Since $L$ is simple, this implies that $\ev\colon\Theta^*\Theta^!L\to L$ is an isomorphism. As every object in $\cA$ is of finite length, every object in $\cA$ is filtered by simple objects. Thus, every object in $\Db(\cA)$ is filtered by shifts of the simple objects in $\cA$. Applying Lemma \ref{devissagetriang} now gives that $\ev\colon\Theta^*\Theta^!\to\id_{\cA}$ is an isomorphism.
A similar argument shows that $\coev\colon\id_{\cA}\to\Theta^!\Theta^*$ is an isomorphism. Hence, $\Theta^*$ and $\Theta^!$ are mutually inverse derived equivalences.
\end{example}

\section{Category $\cO$ and translation functors}\label{s:catO}
\subsection{}
Let $\fg\supset \fb\supset \fh$ be a complex semisimple Lie algebra, a Borel subalgebra and a Cartan subalgebra contained in it, respectively. Let $U(\fg)$ denote the universal enveloping algebra of $\fg$ and let $\fz\subset U(\fg)$ denote the center. Let $\cO$ be the BGG-category $\cO$. That is, $\cO$ consists of all finitely generated $U(\fg)$-modules which are locally finite over $\fb$ and semisimple over $\fh$. For $\lambda\in\fh^*$ let $\Verma{\lambda}= U(\fg)\otimes_{\fb}\CC_{\lambda}$ be the Verma module; here $\CC_{\lambda}$ is the one dimensional $\fh$-module given by $\lambda$ and extended to $\fb$ trivially. Let $\simple{\lambda}$ denote the unique simple quotient of $\Verma{\lambda}$. It is well known (see \cite{BGG}) that every object in $\cO$ has finite length and that if $L\in\cO$ is simple, then $L\simeq \simple{\lambda}$ for some $\lambda\in\fh^*$. Let $\cdot^{\vee}\colon \cO \to \cO$ denote the contravariant duality on $\cO$. Namely, if $M\in\cO$, then $M^{\vee}$ is the vector space of linear functions $M\to \CC$ with finite dimensional support. The $\fg$-action on $M^{\vee}$ is given by the $\fg$-action on $M$ twisted by the Chevalley anti-automorphism. If $L\in\cO$ is simple, then $L^{\vee}\simeq L$. Furthermore, $\cdot^{\vee\vee}\simeq \id$. The modules $\coVerma{\lambda}$ will be referred to as dual Verma modules. 

\subsection{}\label{ss:weyl}
Let $W$ be the Weyl group of $\fg \supset \fb$, let $\ell \colon W \to \ZZ_{\geq 0}$ denote the length function and let $\leq$ denote the Bruhat order on $W$. In particular, $x<y$ means $x\leq y$ and $x\neq y$. The identity element in $W$ is denoted by $e$.
Let $\rho\in\fh^*$ be the half sum of positive roots and let $w_0\in W$ be the longest element of the Weyl group. For $w\in W$ and $\lambda\in \fh^*$ put $w\cdot \lambda = w(\lambda+\rho)-\rho$.

\subsection{}Let $\lambda\in\fh^*$ be integral dominant but perhaps singular. In other words, $\lambda$ is integral and $\lambda+\rho$ lies in the closure of the dominant Weyl chamber. Let $\cO_{\lambda}\subset \cO$ be full subcategory consisting of those objects in $\cO$ whose (generalized) infinitesimal character coincides with the one of $\simple{\lambda}$. That is, those objects which have the same annihilator in $\fz$ as the module $\simple{\lambda}$. For instance, the so called principal block $\cO_0$ consists of objects with trivial infinitesimal character.

\subsection{}Let $s\in W$ be a simple reflection. Let $\lambda\in \fh^*$ be an integral dominant weight such that the stabilizer of $\lambda$ under the `dot-action' of $W$ (see \S\ref{ss:weyl}) is $\{e, s\}$. Let $\pi_{s*}\colon \cO_0 \to \cO_{\lambda}$ be the functor of translation onto the $s$-wall and let $\pi_s^*\colon \cO_{\lambda}\to \cO_0$ be the functor of translation off the $s$-wall. The functor $\pi_s^*$ is both left and right adjoint to $\pi_{s*}$.

\subsection{}Let $x\in W$. To lighten notation we set
\[ \Verma{x} = \Verma{w_0x^{-1}\cdot 0} \qquad\mbox{and} \qquad \simple{x} = \simple{w_0x^{-1}\cdot 0}.\]
The following is well known:
\begin{prop}[{\cite[Satz.\ 2.10(i), Thm.\ 2.11, Satz 2.17]{Ja}}]\label{translationeffect}
Let $s\in W$ be a simple reflection and let $x\in W$. Then
\begin{enumerate}
\item $\pi_s^*\pi_{s*}\Verma{x}$ has a filtration with subquotients isomorphic to $\Verma{x}$ and $\Verma{sx}$ each occuring with multiplicity one.
\item $\pi_s^*\pi_{s*}\coVerma{x}$ has a filtration subquotients isomorphic to $\coVerma{x}$ and $\coVerma{sx}$ each occuring with multiplicity one.
\item If $sx < x$, then $\pi_s^*\pi_{s*}\simple{x} = 0$.
\end{enumerate}
\end{prop}

\subsection{}\label{s:definetheta}
Fix adjunctions $(\pi_s^*, \pi_{s*})$ and $(\pi_{s*}, \pi^*_s)$. Write $\varepsilon$ for the counit of the pair $(\pi_s^*, \pi_{s*})$ and $\eta'$ for the unit of the pair $(\pi_{s*},\pi_s^*)$.
Following \cite[\S4.1.5]{Ro} and \cite{Ri} set
\[ \Theta_s^* = 0\to \pi_s^*\pi_{s*} \mapright{\varepsilon} \id \to 0 \quad \mbox{and}\quad \Theta_s^!= 0\to\id\mapright{\eta'}\pi_s^*\pi_{s*}\to 0, \]
with $\pi_s^*\pi_{s*}$ in degree $0$ in both cases.

\begin{prop}\label{dequivO}The functors $\Theta_s^*$ and $\Theta_s^!$ are mutually inverse self-equivalences of $\Db(\cO_0)$.
\end{prop}

\begin{proof}
Prop.\ \ref{translationeffect} (i) implies that at the level of $K_0(\cO_0)$, $[\pi_s^*\pi_{s*}\Verma w] = [\Verma w] + [\Verma{sw}]$. As the classes of Verma modules give a basis of $K_0(\cO_0)$, we deduce that 
$[\pi_s^*\pi_{s*}\pi_s^*\pi_{s*}X]=2[\pi_s^*\pi_{s*}X]$
for all $X\in\cO_0$. Applying Thm.\ \ref{mainthm} gives the desired result.
\end{proof}

\begin{lemma}\label{adjinjective}Let $s\in W$ be a simple reflection and let $x\in W$ be arbitrary.
\begin{enumerate}
\item The morphism $\eta'\colon  \Verma x\to \pi_s^*\pi_{s*}\Verma x$ is injective.
\item The morphism $\varepsilon\colon \pi^*_s\pi_{s*}\coVerma x \to \coVerma x$ is surjective.
\end{enumerate}
\end{lemma}

\begin{proof}We will only show (i), the proof of (ii) is similar. Prop.\ \ref{translationeffect} (i) implies that $\pi_s^*\pi_{s*}\Verma w$ is non-zero and has a filtration with subquotients isomorphic to Verma modules. According to \cite[Thm.\ 7.6.6]{Dix} any morphism between Verma modules is either $0$ or injective. We infer that $\eta'_s\colon \Verma w\to \pi_s^*\pi_{s*}\Verma w$ is either zero or injective. Lemma \ref{keylem} implies that the map is injective.
\end{proof}

\begin{prop}\label{equivspreservevermas}Let $s\in W$ be a simple reflection and let $M\in \cO_0$.
\begin{enumerate}
\item
If $M$ admits a filtration with subquotients isomorphic to Verma modules, then $\Theta^!_s M$ is in $\cO_0$, i.e., the complex $\Theta^!_s M$ has cohomology concentrated in degree $0$.
\item
If $M$ admits a filtration with subquotients isomorphic to dual Verma modules, then $\Theta^*_s M$ is in $\cO_0$.
\end{enumerate}
\end{prop}

\begin{proof}Follows from Lemma \ref{adjinjective}.
\end{proof}

\begin{prop}\label{behaviourofequivs}Let $s\in W$ be a simple reflection and let $x\in W$.
\begin{enumerate}
\item If $x<sx$, then $\Theta^*_s\Verma{x} \simeq \Verma{sx}$.
\item If $x<sx$, then $\Theta^!_s \coVerma x \simeq \coVerma{sx}$.
\item If $sx<x$, then $\Theta^!_s \simple{x} \simeq \simple{x}[1]$ (or equivalently $\Theta^*_s \simple{x} \simeq \simple{x}[-1]$).
\end{enumerate}
\end{prop}

\begin{proof}
If $x<sx$, then Prop.\ \ref{translationeffect} (i) implies that $\pi_s^*\pi_{s*}\Verma{sx}$ represents a class in $\Ext^1(\Verma{x}, \Verma{sx})$. Using Lemma \ref{adjinjective} we deduce that $\Theta^!_s\Verma{sx}\simeq \Verma{x}$. This gives (i). The proof of (ii) is similar. For (iii), we observe that if $sx<x$, then Prop.\ \ref{translationeffect} (iii) implies $\pi_s^*\pi_{s*}\simple{x} =0$. So $\Theta^!_s \simple{x}\simeq \simple{x}[1]$.
\end{proof}

\begin{thm}[{Bott's Theorem, \cite[Thm.\ 15]{Bott}}]\label{bottsthm}Let $x\in W$ and let $w_0$ be the longest element in $W$. Then
\[\Ext^i(\Verma{x}, \simple{w_0}) = \begin{cases}
\CC & \mbox{if $i=\ell(xw_0)$}, \\
0 & \mbox{otherwise}.
\end{cases}\]
\end{thm}

\begin{proof}
Let $s_1,\ldots, s_m$ be a sequence of simple reflections such that $s_1\cdots s_mx=w_0$ and $\ell(s_i\cdots s_mx)<\ell(s_{i-1}\cdots s_mx)$ for each $1< i < m+1$. That such a sequence exists follows from $w_0$ being the longest element in $W$. Note that $m=\ell(w_0)-\ell(x)=\ell(xw_0)$. So
\begin{align*}
\Ext^i(\Verma{x}, \simple{w_0}) &= \Ext^i(\Theta^!_{s_m}\cdots \Theta_{s_1}^!\Verma{w_0}, \simple{w_0}) \\
&= \Ext^i(\Verma{w_0}, \Theta_{s_1}^*\cdots \Theta_{s_m}^* \simple{w_0}) \\
&= \Ext^{i-\ell(ww_0)}(\Verma{w_0}, \simple{w_0}) \\
&=\begin{cases}
\CC& \mbox{if $i=\ell(xw_0)$}; \\
0 & \mbox{otherwise}.
\end{cases}
\end{align*}
The first equality is a given by Prop. \ref{behaviourofequivs} (i), the second equality is by adjointness and Prop.\ \ref{dequivO}, the third equality is by Prop.\ \ref{behaviourofequivs} (iii) and the final equality is due to the fact that the Verma module $\Verma{w_0}$ is projective in $\cO_0$ (see the first comment in the proof of Prop.\ \ref{behaviourofequivs}).
\end{proof}

\begin{remark}As noted by Bott (see the remarks at the end of \cite{Bott}), the result above gives a realization of the Weyl character formula in $K_0(\cO_0)$: $[\simple{w_0}]=\sum_{x\in W}(-1)^{\ell(xw_0)}[\Verma{x}]$.
\end{remark}

\section{Tilting modules and Soergel's character formula}\label{s:lastnongraded}
\subsection{}For each $w\in W$ fix a reduced word $w=s\cdots t$. Set
\[ \Theta_w^* = \Theta_s^*\cdots \Theta_t^* \quad \mbox{and} \quad \Theta_w^! = \Theta_s^!\cdots\Theta^!_t.\]
Up to natural isomorphism, the $\Theta_w^*$, $\Theta^!_w$ are independent of the choice of reduced word:

\begin{thm}[{\cite[Thm.\ 4.4]{Ro}}]\label{braidrelsO}
Let $w,w'\in W$. If $\ell(ww') = \ell(w) + \ell(w')$, then \[\Theta^*_w\Theta_{w'}^*\simeq \Theta^*_{ww'}.\]
\end{thm}

\begin{prop}\label{lem1}Let $w\in W$. 
\begin{enumerate}
\item
$\Theta_w^* \Verma{e} \simeq \Verma{w}$.
\item
$\Theta^!_w \Verma{e} \simeq \coVerma{w}$.
\end{enumerate}
\end{prop}

\begin{proof}
Let $w=s\cdots t$ be a reduced word. Then $\Theta^*_w\simeq\Theta^*_s\cdots \Theta^*_t$ by Thm.\ \ref{braidrelsO}. Hence, by Prop.\ \ref{behaviourofequivs} (i),
\[ \Theta^*_w \Verma{e} \simeq \Theta^*_s \cdots \Theta^*_t \Verma{e} \simeq \Verma{s\cdots t}=\Verma{w}. \]
This proves (i). The proof of (ii) is analogous (note that $\Verma{e}=\coVerma{e}$).
\end{proof}

\begin{lemma}\label{longelementlemma}Let $x\in W$ and let $w_0$ be the longest element in $W$.
\begin{enumerate}
\item $\Theta_{w_0}^*\coVerma{x}\simeq \Verma{w_0x}$.
\item $\Theta^!_{w_0}\Verma{x}\simeq \coVerma{w_0x}$.
\end{enumerate}
\end{lemma}

\begin{proof}We have 
\[
\Theta_{w_0}^*\coVerma x \simeq \Theta_{w_0}^*\Theta^!_x \Verma e
\simeq \Theta^*_{w_0x}\Theta^*_{x^{-1}}\Theta^!_x \Verma e \simeq \Theta_{w_0x}^*\Verma e \simeq \Verma{w_0x}.
\]
The first isomorphism is Prop.\ \ref{lem1} (ii), the second isomorphism follows from Thm.\ \ref{braidrelsO}, the third isomorphism follows from Prop.\ \ref{dequivO} and the last isomorphism is Prop.\ \ref{lem1} (i). This proves (i). Using Prop.\ \ref{dequivO} we deduce that $(\Theta^*_{w_0})^{-1}=\Theta^!_{w_0}$. Thus, (ii) follows from (i).
\end{proof}

\begin{prop}\label{triangdeltafiltr}Let $X\in\cO_0$, then, as an object of $\Db(\cO_0)$, $X$ is filtered by objects of the form $\Verma{x}[i]$, $i\geq 0$, $x\in W$.
\end{prop}

\begin{proof}
Let $\cO_{\Delta}$ be the subcategory of $\Db(\cO_0)$ consisting of objects $\Verma{x}[i]$, $i\in\ZZ_{\geq 0}$, $x\in W$. 
We will use the notation introduced in \S\ref{s:trianggen}.
If $M\in\cO_{\Delta}^{*\infty}$, then $M[i]\in\cO_{\Delta}^{*\infty}$ for all $i\in\ZZ_{\geq 0}$.
It suffices to show that $\cO_0 \subset \cO_{\Delta}^{*\infty}$. Since every object in $\cO_0$ has finite length, this reduces to showing that each $\simple{x}$, $x\in W$, is in $\cO_{\Delta}^{*\infty}$. Proceed by induction on the length of $x$. If $\ell(x)=0$, then $x=e$ and $\simple{x}=\simple{e}=\Verma{e}$ which is clearly in $\cO_{\Delta}^{*\infty}$. Now let $x\in W$ and assume that if $\ell(x')<\ell(x)$, then $\simple{x'}\in\cO_{\Delta}^{*\infty}$. Let $N_x$ be the kernel of the map $\Verma{x}\twoheadrightarrow \simple{x}$. Then the exact sequence $0 \to N_x \to \Verma{x} \to \simple{x} \to 0$
gives a distinguished triangle
$\Verma{x} \to \simple{x} \to N_x[1] \leadsto$
in $\Db(\cO_0)$. By the induction hypothesis $N_x[1]\in \cO_{\Delta}^{*\infty}$. Consequently, $\simple{x}\in \cO_{\Delta}^{*\infty}$.
\end{proof}

\subsection{}
For each $x\in W$ there exists a unique (up to isomorphism) indecomposable object, denoted $D_x$, characterized by the following properties:
\begin{enumerate}
\item $D_x$ admits a filtration
$0=V_0 \subset V_1 \subset\cdots \subset V_k=D_x$
such that each $V_i/V_{i-1}$ is isomorphic to a dual Verma module and $V_k/V_{k-1}\simeq \coVerma{x}$.
\item $\Ext^i(D_x, \coVerma{y}) = 0$ for all $i\neq 0$ and $y\in W$.
\end{enumerate}
The $D_x$ are the so-called indecomposable tilting modules. They are self-dual, i.e., $D_x^{\vee}\simeq D_x$. See \cite[\S5]{So98} for a streamlined treatment of tilting modules.

\subsection{}It is well known (see \cite[\S4]{BGG}) that category $\cO$ has enough projectives. 
For $\lambda\in\fh^*$ let $P_{\lambda}$ denote the indecomposable projective cover of $\simple{\lambda}$. Further, for $x\in W$ let $P_x$ denote the indecomposable projective cover of $L_x$ and set $I_x = P_x^{\vee}$.

The following result is the category $\cO$ analogue of \cite[Thm.\ 6.10]{BeGi} ($D$-modules) and \cite[\S2.3]{BBM} (perverse sheaves). The proof presented here is formally the same as that of \cite[Prop.\ 2.3]{BBM}, also see \cite[Thm.\ 8]{StM}. Actually, the Radon transforms of \cite{BBM} are Koszul dual (in the sense of \cite{BGS}) to the $\Theta^*_w$.
\begin{thm}\label{switchtilting}Let $x\in W$ and let $w_0$ be the longest element in $W$. Then
\begin{enumerate}
\item
$\Theta^*_{w_0}D_x \simeq P_{w_0x}$;
\item
$\Theta^*_{w_0}I_x \simeq D_{w_0x}$.
\end{enumerate}
\end{thm}

\begin{proof}
We will only prove (i), the proof of (ii) is similar.
Since $D_x$ has a dual Verma filtration, Prop.\ \ref{equivspreservevermas} (ii) implies that $\Theta^*_{w_0}D_x$ lies in $\cO_0$. Let $y\in W$ and let $i>0$, then
\[\Ext^i_{\cO_0}(\Theta_{w_0}^* D_x, \Verma{y}) = \Ext_{\cO_0}^i(D_x, \Theta_{w_0}^! \Verma{y} )
= \Ext_{\cO_0}^i(D_x, \coVerma{w_0y})
= 0.
\]
The first equality is given by Prop.\ \ref{dequivO} and Thm.\ \ref{braidrelsO}. The second equality is Lemma \ref{longelementlemma} (ii) and the last equality is by the definition of $D_x$.
Combining this with Lemma \ref{triangdeltafiltr} we deduce that if $i>0$, then $\Ext^i_{\cO_0}(\Theta_{w_0}^* D_x, X)=0$ for all $X\in\cO_0$. Thus $\Theta_{w_0}^*D_x$ is projective. Since $D_x$ is indecomposable and $\Theta_{w_0}^*$ is an equivalence, we deduce that $\Theta_{w_0}^* D_x$ is indecomposable. It remains to show that $\Theta_{w_0}^* D_x$ surjects onto $\simple{w_0x}$. As $D_x$ is self-dual,
Lemma \ref{longelementlemma} (i) implies that $\Theta_{w_0}^*D_x$ surjects onto $\Verma{w_0x}$. Thus, $\Theta_{w_0}^*D_x$ surjects onto $\simple{w_0x}$. 
\end{proof}

\begin{cor}[{\cite[Thm.\ 6.7]{So98}}]\label{tiltingcharformula}Let $x,y\in W$ and let $w_0$ be the longest element in $W$. Then, at the level of the Grothendieck group $K_0(\cO_0)$:
\[ [ D_x : \Verma{y} ] = [P_{w_0x} : \Verma{w_0y}] = [\Verma{w_0y} : \simple{w_0x}]. \]
\end{cor}
\begin{proof}Working in $K_0(\cO_0)$, we have
\[
[D_x : \Verma{y} ] = [\Theta^*_{w_0} D_x : \Theta^*_{w_0}\Verma{y}] 
= [P_{w_0x} : \coVerma{w_0y}] 
= [\Verma{w_0y} : \simple{w_0x}]. \]
The first equality is a consequence of Prop.\ \ref{dequivO}. The second equality is obtained from Thm.\ \ref{ringelselfduality} (i) and by combining Lemma \ref{longelementlemma} (ii) with the fact that at the level of $K_0(\cO_0)$, $[\Theta^*_sX]= [\Theta^!_s X]$ for all $X\in\cO_0$ and each simple reflection $s\in W$. The last equality is BGG reciprocity (see \cite[\S6 Prop.\ 2]{BGG}).
\end{proof}

\begin{cor}[\cite{So98}]\label{ringelselfduality}$\bigoplus_{x\in W} \End(P_x) \simeq \bigoplus_{x\in W}\End(D_x)$.
\end{cor}
\begin{proof}Let $w_0$ be the longest element in $W$. Then $w_0^{-1}=w_0$. Thus, Prop.\ \ref{dequivO} gives that $(\Theta^*_{w_0})^{-1}=\Theta^!_{w_0}$. So, by Thm.\ \ref{ringelselfduality} (i), we have
\[ \bigoplus_{x\in W} \End(P_x)\simeq\bigoplus_{x\in W} \End(\Theta_{w_0}^!P_x)\simeq\bigoplus_{w\in W}\End(D_x). \qedhere \]
\end{proof}

\section{Complements on graded category $\cO$}\label{s:gO}
We start by reviewing some ideas of Soergel and Stroppel.
\subsection{}In the following graded will always mean $\ZZ$-graded. Modules over an algebra will mean right modules.
Let $A$ be a finite dimensional graded $\CC$-algebra. Let $A\mof$ be the category of all finite dimensional $A$-modules and let $A\gmof$ be the category of all graded finite dimensional $A$-modules. Denote by $\Hom_A(-,-)$ (resp.\ $\Hom_{A^{\mathrm{gr}}}(-,-)$) the morphisms in $A\mof$ (resp.\ $A\gmof$). Let $\nu\colon A\gmof \to A\mof$ be the functor of forgetting the grading. This is a faithful functor. Let $M=\bigoplus_{i\in\ZZ} M_i$ be a graded $A$-module with $M_i$ the component of degree $i$. For $n\in\ZZ$, define $M\langle n \rangle$ by $M\langle n \rangle_i = M_{i-n}$. Thus, $\nu M\langle n \rangle = \nu M$ and $\Hom_A(\nu M, \nu N) = \bigoplus_{n\in\ZZ}\Hom_{A^{\mathrm{gr}}}(M\langle n \rangle, N)$, $M,N\in A\gmof$.

Let $M\in A\mof$. Suppose there is a $\tilde M \in A\gmof$ such that $\nu \tilde M = M$, then we say that $\tilde M$ is a lift of $M$.\begin{lemma}Any two lifts of an indecomposable module $M\in A\mof$ are isomorphic up to grading shift.
\end{lemma}

\begin{proof}Let $M', M''$ be two lifts of $M$. Then the identity map $M\to M$ in
\[ \Hom_A(M, M) = \bigoplus_{n\in\ZZ}\Hom_{A^{\mathrm{gr}}}(M'\langle n \rangle, M'')\] decomposes into homogeneous components $\id = \sum_n \id_n$. By the Fitting Lemma, $\Hom_A(M, M)$ is a local ring. Thus, $\id_j$ must be invertible for some $j$.
\end{proof}

\begin{prop}\label{projprop}Let $P\in A\mof$ be an indecomposable projective.
Then any lift of $P$ is an indecomposable projective in $A\gmof$.
\end{prop}

\begin{proof}Let $\tilde P$ be a lift of $P$. Let $0\to M\to N\mapright{f} \tilde P \to 0$ be an exact sequence in $A\gmof$. As $\nu \tilde P = P$ is projective, there exists $g\in \Hom_A (P, \nu N)$ such that $fg = \id_P$. Let $g = \sum_i g_i$ be the decomposition of $g$ into homogeneous components corresponding to the decomposition
$\Hom_A(P, \nu N) = \bigoplus_{n\in\ZZ}\Hom_{A^{\mathrm{gr}}}(\tilde P\langle n \rangle, N)$.
By the Fitting Lemma, $\End_A(P)$ is a local ring. Hence, $fg_j$ is invertible for some $j$. Let $h\in \Hom_{A^{\mathrm{gr}}}(\tilde P\langle -j\rangle, \tilde P )$ denote the inverse of $fg_j$, then $g_jh$ is homogeneous of degree $0$ and $fg_jh=\id_{\tilde P}$. Thus, $P$ is projective. That it is indecomposable is clear.
\end{proof}

\subsection{}
Let $S=S(\fh)$ denote be the algebra of regular functions on $\fh^*$. 
We consider $S$ as an evenly graded algebra with linear functions in degree $2$.
Let $S_+\subset S$ denote the maximal ideal consisting of functions that vanish at $0$. Let $S_+^W\subset S_+$ be the sub-ideal consisting of $W$-invariant (regular action) functions in $S_+$. Set $C=S/S_+^W$, then $C$ is the so-called coinvariant algebra of $W$.

Let $\lambda\in\fh^*$ be integral dominant. Let $W_{\lambda}\subseteq W$ denote the stabilizer of $\lambda$ under the dot action (see \S\ref{ss:weyl}).
\begin{thm}[{\cite[Endomorphismensatz 7]{So90}}]
There is an isomorphism of algebras
\[ \End_{\fg}(\projective{w_0\cdot \lambda}) \simeq C^{\lambda}, \]
where $C^{\lambda}$ denotes the subalgebra of $W_{\lambda}$-invariants in $C$.
\end{thm}
\subsection{}
Define
\[ \VV\colon \cO_{\lambda} \to C^{\lambda}\mof, \quad M\mapsto \Hom_{\fg}(\projective{w_0\cdot \lambda}, M).\]
\begin{thm}[{\cite[Struktursatz 9]{So90}}]The functor $\VV$ is full and faithful on projective objects.
\end{thm}

\subsection{}\label{s:gradedrings}Certainly $C$ and $C^{\lambda}$ inherit a grading from $\fh$. According to \cite[Thm.\ 2.1]{St}, if $P\in \cO_{\lambda}$ is projective, then $\VV P$ admits a lift. Let $[W/W_{\lambda}]$ denote the set of minimal length coset representatives of $W/W_{\lambda}$. For each $x\in [W/W_{\lambda}]$, let $\widetilde{\VV \projective{x\cdot \lambda}}$ be a fixed lift of $\VV \projective{x\cdot\lambda}$ with highest non-zero component in degree $\ell(x)$. Set
\begin{align*}
A_{\lambda} &= \End_{\fg}(\bigoplus_{x\in [W/W_{\lambda}]} \projective{x\cdot\lambda}) \\
&= \End_{C^{\lambda}}(\bigoplus_{x\in W/W_{\lambda}}\VV \projective{x\cdot \lambda}) \\
&= \bigoplus_{n\in\ZZ}\Hom_{(C^{\lambda})^{\mathrm{gr}}}(\bigoplus_{x\in [W/W_{\lambda}]}\widetilde{\VV P_{x\cdot \lambda}}\langle n\rangle,\bigoplus_{x\in [W/W_{\lambda}]}\widetilde{\VV P_{x\cdot \lambda}}).\end{align*}
In particular, $A_{\lambda}$ is a graded ring. Furthermore, as $\bigoplus_{x\in [W/W_{\lambda}]} \projective{x\cdot \lambda}$ is a minimal projective generator of $\cO_{\lambda}$, there is an equivalence of categories
\[ \cO_{\lambda} \mapright{\sim} A_{\lambda}\mof, \quad M\mapsto \Hom_{A_{\lambda}}(\bigoplus_{x\in [W/W_{\lambda}]} \projective{x\cdot \lambda}, M).\]
We will not distinguish between $\cO_{\lambda}$ and $A_{\lambda}\mof$. If $\lambda=0$, we simply write $A$ instead of $A_0$. Set $\cO_{\lambda}^{\ZZ} = A_{\lambda}\gmof$.

\begin{thm}[{\cite[Thm.\ 8.1, Thm.\ 8.2]{St}}]\label{stadjoints}Let $s\in W$ be a simple reflection. The translation functors $\pi_{s*}$ and $\pi_s^*$ are gradable. More precisely, there exist functors $\theta_0^{\lambda}\colon \cO_0^{\ZZ} \to \cO_{\lambda}^{\ZZ}$ and $\pi^0_{\lambda}\colon \cO_{\lambda}^{\ZZ}\to \cO_0^{\ZZ}$ that commute with grading shifts and are such that the following diagrams commute
\[\xymatrix{
\cO_0^{\ZZ}\ar[r]^-{\theta_0^{\lambda}}\ar[d]_-{\nu} & \cO_{\lambda}^{\ZZ}\ar[d]^-{\nu} \\
\cO_0\ar[r]^-{\pi_{s*}} & \cO_{\lambda}
}\qquad
\xymatrix{
\cO_{\lambda}^{\ZZ}\ar[r]^-{\theta_{\lambda}^0}\ar[d]_-{\nu} & \cO_{0}^{\ZZ}\ar[d]^-{\nu} \\
\cO_{\lambda}\ar[r]^-{\pi_{s}^*} & \cO_{0}
}
\]
(here $\lambda$ is an integral dominant weight with stabilizer $\{e,s\}$).
\end{thm}

\begin{thm}[{\cite[Thm.\ 8.4]{St}}]
The functor $\theta^0_{\lambda}$ is left adjoint to $\theta_0^{\lambda}\langle -1\rangle$ and the functor $\theta^0_{\lambda}\langle -1\rangle$ is right adjoint to $\theta_{0}^{\lambda}$.
\end{thm}

\begin{warning}There is a misprint in \cite[Thm.\ 8.4]{St}. The result therein states that $\theta_{\lambda}^0$ is left adjoint to $\theta_0^{\lambda}\langle 1 \rangle$. However, examining its proof, we have $\Hom_{\cO_0^{\ZZ}}(\theta_{\lambda}^0M, N) \simeq \Hom_{\cO_0^{\ZZ}}(M, N\otimes W^{\circledast})$. Where, in the notation of \cite{St}, $W^{\circledast}= \Hom_{C^{\lambda}}(\VV P_{\lambda}, \mathrm{res}\, \VV P)\langle -1\rangle$ (see two lines above \cite[Cor.\ 8.5]{St}). Further, $\theta_0^{\lambda}=-\otimes\Hom_{C^{\lambda}}(\VV P_{\lambda}, \mathrm{res}\, \VV P)$ in \cite{St} (see \cite[Thm.\ 8.1]{St}).
\end{warning}

\subsection{}We now work mainly with the principal block, i.e., the categories $\cO_0$ and $\cO_0^{\ZZ}$. For each $x\in W$, set
\[ \gprojective{x} = \bigoplus_{n\in\ZZ}\Hom_{C^{\mathrm{gr}}}(\bigoplus_{x\in W} \widetilde{\VV\projective{x}}\langle n\rangle, \widetilde{\VV \projective{x}}).\]
By definition, $\gprojective{x}\in\cO^{\ZZ}_0$ is a lift of $\projective{x}$; by Prop.\ \ref{projprop}, each $\gprojective{x}$ is an indecomposable projective in $\cO^{\ZZ}_0$. Let $\gsimple{x}$ denote the unique irreducible quotient of $\gprojective{x}$. Certainly $\nu \gsimple{x}$ is irreducible, we deduce that $\gsimple{x}$ is a lift of $\simple{x}$. By \cite[Thm.\ 2.1]{St}, the $\gsimple{x}$ are concentrated in degree $0$. Finally, according to \cite[\S3.3]{St}, Verma modules admit lifts. We let $\gVerma{x}$ denote the lift of $\Verma{x}$ that has $\gsimple{x}$ as its unique simple quotient.
\begin{warning}Not all objects of $\cO$ lift, see \cite[\S4]{St}.
\end{warning}

\subsection{}Let $s$ be a simple reflection and let $\theta_0^{\lambda}$ and $\theta_{\lambda}^0$ be as in Thm.\ \ref{stadjoints}. Let $\theta_s = \theta_{\lambda}^0\theta_0^{\lambda}$.
\begin{thm}[{\cite[Thm.\ 3.6, Thm.\ 5.3]{St}}]\label{ses}Let $x\in W$.
\begin{enumerate}
\item If $sx<x$, then there is a short exact sequence
\[ 0 \to \gVerma{x}\langle 1 \rangle \to \theta_s \gVerma{x} \to \gVerma{sx} \to 0.\]
\item If $x<sx$, then there is a short exact sequence
\[ 0 \to \gVerma{sx} \to\theta_s\gVerma{x} \to \gVerma{x}\langle -1\rangle \to 0.\]
\end{enumerate}
\end{thm}

\subsection{}Set 
\[ \pi_{s*} = \theta_0^{\lambda}, \quad \pi_s^*=\theta_{\lambda}^0\langle 1 \rangle, \quad \pi_s^! = \theta_{\lambda}^0\langle -1\rangle. \]
Then we have adjunctions $(\pi_s^*, \pi_{s*})$ and $(\pi_{s*}, \pi_s^!)$. Note that $\pi_s^! = \pi_s^*\langle -2\rangle$.
Let $\eta'$ be the unit of $(\pi_{s*}, \pi_s^!)$ and let $\varepsilon$ be the counit of $(\pi_s^*, \pi_{s*})$. Define complexes of functors
\[ \TT_s = 0 \to \pi_s^*\pi_{s*}\mapright{\varepsilon} \id \to 0 \quad \mbox{and} \quad \TT_s^{-1} = 0 \to \id \mapright{\eta'} \pi_s^!\pi_{s*} \to 0, \]
with $\pi_s^*\pi_{s*}$ (resp.\ $\pi_s^!\pi_{s*}$) in cohomological degree $0$. It is straightforward to verify that there are natural isomorphisms $\nu \TT_s \simeq \Theta_s^*\nu$ and $\nu\TT_s^{-1} \simeq \Theta_s^! \nu$ (see \cite[Prop.\ 2.2]{Bass}).

\begin{thm}\label{dequivOgraded}The functors $\TT_s$ and $\TT_s^{-1}$ are mutually inverse equivalences of $\Db(\cO_0^{\ZZ})$.
\end{thm}

\begin{proof}Let $x\in W$. Using Thm.\ \ref{ses} we compute in $K_0(\cO^{\ZZ}_0)$: if $sx<x$, then $[\pi^*_s\pi_{s*}\gVerma{x}] = [\gVerma{x}\langle 2\rangle] + [\gVerma{sx}\langle 1\rangle]$
\[
[\pi^!_s\pi_{s*}\pi_s^*\pi_{s*}\gVerma{x}] = [\gVerma{x}\langle 2 \rangle]+ [\gVerma{sx}\langle 1 \rangle]+ [\gVerma{x}] + [\gVerma{sx}\langle -1 \rangle]
= [\pi^*_s\pi_{s*}\gVerma{x}] + [\pi^!_s\pi_{s*}\gVerma{x}].
\]
If $sx>x$, then $[\pi^*_s\pi_{s*}\gVerma{x}]= [\gVerma{sx}\langle 1\rangle] + [\gVerma{x}]$.
\[
[\pi^!_s\pi_{s*}\pi_s^*\pi_{s*}\gVerma{x}] = [\gVerma{sx}\langle 1\rangle] + [\gVerma{x}] + [\gVerma{sx}\langle -1\rangle] + [\gVerma{x}\langle -2\rangle]
= [\pi_s^*\pi_{s*}\gVerma{x}] + [\pi_s^!\pi_{s*}\gVerma{x}].
\]
Further, $\pi^!_s\pi_{s*}\pi_s^*\pi_{s*}=\pi_s^*\pi_{s*}\pi_s^!\pi_{s*}$. As the graded Verma modules $\gVerma{x}\langle n\rangle$, $x\in W$, $n\in\ZZ$, constitute a basis of $K_0(\cO_0^{\ZZ})$ we deduce that we are in the situation of Thm.\ \ref{mainthm}. Consequently, $\TT_s$ and $\TT_s^{-1}$ are mutually inverse equivalences.
\end{proof}

\subsection{}By \cite[\S6]{St} there is a `graded duality' $\DD\colon \cO_0^{\ZZ} \to \cO_0^{\ZZ}$. The functor $\DD$ is contravariant, commutes with reflection across the wall (i.e., $\DD\theta_s \simeq \theta_s\DD$) and satisfies the following:
\[ \DD^2 \simeq \id, \quad \DD(M\langle n \rangle) \simeq (\DD M) \langle -n \rangle, \quad \nu(\DD M) \simeq (\nu M)^{\vee}, \quad \DD\gsimple{x} \simeq \gsimple{x}, \]
for all $M\in \cO_0^{\ZZ}$, $n\in \ZZ$ and $x\in W$. We set $\gcoVerma{x} = \DD \gVerma{x}$. It is clear that $\gcoVerma{x}$ is a lift of $\coVerma{x}$.

\begin{lemma}\label{gradedadjinjective}Let $s\in W$ be a simple reflection and let $x\in W$ be arbitrary.
\begin{enumerate}
\item The morphism $\eta'\colon \gVerma{x} \to \pi_s^!\pi_{s*}\gVerma{x}$ is injective.
\item The morphism $\varepsilon\colon \pi_s^*\pi_{s*}\gcoVerma{x}\to \gcoVerma{x}$ is surjective.
\end{enumerate}
\end{lemma}
\begin{proof}Left to the reader (see Lemma \ref{adjinjective}).
\end{proof}

\begin{prop}\label{gbehaviourofequivs}Let $s\in W$ be a simple reflection and let $x\in W$.
\begin{enumerate}
\item If $x<sx$, then $\TT_s\gVerma{x}\langle -1\rangle \simeq \gVerma{sx}$.
\item If $x<sx$, then $\TT_s^{-1}\gcoVerma{x}\langle 1 \rangle \simeq \gcoVerma{sx}$.
\item If $sx<x$, then $\TT_{s}^{-1}\gsimple{x} \simeq \gsimple{x}[1]$ (or equivalently $\TT_s \gsimple{x} \simeq \gsimple{x}[1]$).
\end{enumerate}
\end{prop}

\begin{proof}This is proved in exactly the same way as Prop.\ \ref{behaviourofequivs}. If $x<sx$, then by Thm.\ \ref{ses}(i), the object $\pi_s^!\pi_{s*}\gVerma{sx}=\theta_s \gVerma{sx}\langle -1 \rangle$ represents a class in $\Ext^1(\gVerma{x}\langle -1\rangle,\gVerma{sx})$. Using Lemma \ref{gradedadjinjective} (i) we deduce that $\TT_s^{-1}\gVerma{sx}\simeq \gVerma{x}\langle -1 \rangle$. This shows (i). For (ii), we have 
\[ \pi_{s}^*\pi_{s*}\gcoVerma{sx}=\theta_s\gcoVerma{sx}\langle 1 \rangle) = \theta_s\DD(\gVerma{sx}\langle -1\rangle)=\DD(\theta_s\gVerma{sx}\langle -1\rangle). \]
So, applying $\DD$ to Thm.\ \ref{ses} (i), we deduce that $\pi_{s}^*\pi_{s*}\gcoVerma{sx}$ represents a class in $\Ext^1(\gcoVerma{sx}, \gcoVerma{x}\langle 1\rangle))$. Using Lemma \ref{gradedadjinjective} (ii) we obtain that $\TT_s\gcoVerma{sx} \simeq \gcoVerma{x}\langle 1\rangle$. This proves (ii). If $sx<x$, then by Thm.\ \ref{stadjoints} and Prop.\ \ref{translationeffect} (iii), we have that $\nu\pi_s^*\pi_{s*}\gsimple{x} =0$. Thus, $\pi_{s}^*\pi_{s*}\gsimple{x} =0$. This implies (iii).
\end{proof}

\begin{prop}[{cf.\ Thm.\ \ref{bottsthm}}]
Let $n\in \ZZ$, $x\in W$ and let $w_0$ be the longest element in $W$. Then
\[ \Ext^i(\gVerma{x}\langle n \rangle ,\gsimple{w_0}) = \begin{cases}
\CC &\mbox{if $i=\ell(ww_0)$ and $n=-\ell(ww_0$);} \\
0 &\mbox{otherwise}.
\end{cases}\]
\end{prop}

\begin{proof}
This is proved in exactly the same way as Thm.\ \ref{bottsthm}.
Let $s_1,\ldots, s_m$ be a sequence of simple reflections such that $s_1\cdots s_mw=w_0$ and $\ell(s_i\cdots s_mw_0)< \ell(s_{i-1}\cdots s_m w_0)$ for each $1< i < m+1$. Note that $m=\ell(w_0)-\ell(w) = \ell(ww_0)$. We have
\begin{align*}
\Ext^i(\gVerma{x}\langle n \rangle, \gsimple{w_0}) & = \Ext^i(\TT_{s_m}^{-1} \cdots \TT_{s_1}^{-1}\gVerma{w_0}\langle m+n\rangle, \gsimple{w_0}) \\
&= \Ext^i(\gprojective{w_0}\langle m+n\rangle, \TT_{s_1}\cdots \TT_{s_m}\gsimple{w_0}) \\
&= \Ext^{i-\ell(ww_0)}(\gprojective{w_0}\langle m+n \rangle, \gsimple{w_0}) \\
&=\begin{cases}
\CC &\mbox{if $i=\ell(ww_0)$ and $n=-\ell(ww_0$);} \\
0 &\mbox{otherwise}.
\end{cases}\qedhere
\end{align*}
\end{proof}

\subsection{}For each $w\in W$ fix a reduced word $w=s\cdots t$. Set
\[ \TT_{w} = \TT_s \cdots \TT_t \quad \mbox{and} \quad \TT_w^{-1}= \TT_t^{-1} \cdots \TT_s^{-1}.\]
\begin{thm}[\cite{Ro}]\label{gradedbraid}Let $w,w'\in W$. If $\ell(ww')=\ell(w)+\ell(w')$, then
\[ \TT_w\TT_{w'} \simeq \TT_{ww'}.\]
\end{thm}
\begin{proof}This follows from \cite[Prop.\ 3.2]{Ro}, since all the isomorphisms in \emph{loc. cit.} are of complexes of \emph{graded} bimodules.
\end{proof}

\subsection{}\label{s:gtilting}Let $w_0$ be the longest element in $W$. For each $x\in W$, set
\begin{equation}\label{definegtilt} \gtilting{x} = \TT_{w_0}^{-1}\gprojective{w_0x}\langle \ell(w_0)\rangle.\end{equation}
Then $\gtilting{x}$ is a lift of the tilting module $D_x$. Further, using Prop.\ \ref{gbehaviourofequivs} we deduce that
\begin{align*}
\TT_{w_0}\gcoVerma{x} &\simeq \TT_{w_0}\TT^{-1}_{x^{-1}}\gVerma{e}\langle \ell(x)\rangle\\
&\simeq \TT_{w_0x}\TT_{x^{-1}}\TT^{-1}_{x^{-1}}\gVerma{e}\langle \ell(x) \rangle\\
&\simeq \TT_{w_0x}\gVerma{e}\langle \ell(x) \rangle \\
&\simeq \gVerma{w_0x} \langle \ell(x) + \ell(w_0x)\rangle \\
&= \gVerma{w_0x} \langle \ell(w_0)\rangle.
\end{align*}
Thus, $\gtilting{x}$ is the unique (up to isomorphism) indecomposable object in $\cO_0^{\ZZ}$ satisfying the following properties:
\begin{enumerate}
\item $\gtilting{x}$ admits a filtration
$0 = V_0 \subset V_1 \subset \cdots \subset V_k = \gtilting{x}$
such that $V_i/V_{i-1}$ is isomorphic to the shift of a graded dual Verma module and $V_k/V_{k-1}\simeq \gcoVerma{x}$.
\item $\Ext^i(\gtilting{x}, \gcoVerma{y}\langle n \rangle)=0$ for all $i\neq 0$, $n\in \ZZ$ and $y\in W$.
\end{enumerate}

\begin{prop}\label{tiltselfdual}The modules $\gtilting{x}$ are self-dual, i.e., $\DD\gtilting{x}\simeq \gtilting{x}$.
\end{prop}

\begin{proof}
As $\nu\DD \gtilting{x} \simeq (\nu\gtilting{x})^{\vee} = D_x^{\vee} \simeq D_x$, we must have $\DD\gtilting{x} \simeq \gtilting{x}\langle n \rangle$ for some $n\in\ZZ$. Now $\gsimple{x}\langle m\rangle$ occurs as a subquotient of $\gtilting{x}\langle n\rangle$ if and only if $m=n$. On the other hand $\DD \gcoVerma{x} \simeq \gVerma{x}$ occurs as a submodule of $\gtilting{x}$. The result follows.
\end{proof}

\section{Complements on Kazhdan-Lusztig theory}\label{s:kl}
\subsection{}The Hecke algebra $\cH$ is the free $\ZZ[v,v^{-1}]$-module $\bigoplus_{x\in W}\ZZ[v,v^{-1}]H_x$ with $\ZZ[v,v^{-1}]$-algebra structure given by
\begin{align}
H_xH_y &= H_{xy} && \mbox{if $\ell(xy)=\ell(x) + \ell(y)$}, \label{braidh}\\
(H_s+v)(H_s-v^{-1}) &= 0 && \mbox{if $s$ is a simple reflection}. \label{quadh}
\end{align}

\subsection{}
There is a unique ring automorphism $d\colon \cH\to\cH$ defined by
\[ d(v)=v^{-1}, \quad d(H_x)=H_{x^{-1}}^{-1}. \]
An element $C\in\cH$ is called self dual if $d(C)=C$. For each $x\in W$ there exists a unique self-dual element $C_x$ such that $C_x\in H_x + \sum_{y}v\ZZ[v]H_y$ (see \cite{KL}).

\subsection{}
Let $b\colon \cH\to\cH$ be the ring automorphism defined by
\[ b(v) = -v^{-1}, \quad b(H_x)=H_x.\]
Then $b$ commutes with $d$. Thus, $C'_x=b(C_x)$ is the unique self-dual element such that $C'_x\in H_x + \sum_y v^{-1}\ZZ[v^{-1}]H_y$.

\subsection{}Consider the Grothendieck group $K_0(\cO_0^{\ZZ})$. For $[X]\in K_0(\cO_0^{\ZZ})$, set
\[ v^n[X]=[X\langle -n \rangle], \quad H_x[X] = [\TT_xX\langle -\ell(x) \rangle].\]
This defines an action of $\cH$ on $K_0(\cO_0^{\ZZ})$. The relations \eqref{braidh} follow from Thm.\ \ref{gradedbraid}. To see \eqref{quadh}, let $s$ be a simple reflection, then
\begin{align*}
[(\TT_s\langle -1\rangle + \id\langle -1\rangle)(\TT_s\langle -1 \rangle - \id\langle 1 \rangle)X]
&=[(\pi_s^*\pi_{s*}\langle -1\rangle)(\TT_s\langle -1 \rangle - \id\langle 1 \rangle)X] \\
&=[(\pi_s^!\pi_{s*}\langle 1 \rangle)(\TT_s\langle -1 \rangle - \id\langle 1 \rangle)X] \\
&=[(\TT_s^{-1}\langle 1 \rangle + \id\langle 1 \rangle)(\TT_s\langle -1 \rangle - \id\langle 1 \rangle)X] \\
&=[(\id - \TT_{s}^{-1}\langle 2 \rangle + \TT_s - \id\langle 2\rangle)X]\\
&=[(\id - \pi_s^*\pi_{s*} + \id\langle 2 \rangle + \pi_s^*\pi_{s*}-\id - \id\langle 2\rangle)X]\\
&=0.
\end{align*}

\subsection{}The map $\phi\colon \cH \to K_0(\cO_0^{\ZZ})$, $H\mapsto H[\gVerma{e}]$ defines an isomorphism of left $\cH$-modules. By Prop.\ \ref{gbehaviourofequivs} we have that $\phi(H_x) = [\gVerma{x}]$. Further, the map $\phi$ intertwines the automorphism $d$ and the contravariant duality, i.e., $\phi(d(H)) = \DD\phi(H)$.
Thus, self-dual elements in $\cH$ map to elements $[L]\in K_0(\cO_0^{\ZZ})$ such that $[\DD L]=[L]$. Certainly, $[\DD \gsimple{x}] = [\gsimple{x}]$. The Kazhdan-Lusztig conjecture (a Theorem since about 30 years), concerning multiplicities of simple modules in Verma modules \cite{KL}, can be formulated as
\begin{equation}\label{kleq}\tag{*}\phi^{-1}([\gsimple{x}]) = b(C_x).\end{equation}
Unfortunately (but not surprisingly), the work we have done so far does not give enough information to prove this. The problem is that although the $\phi^{-1}([\gsimple{x}])$ are self dual, we do not have enough information to infer
\begin{equation}\label{uppertriang}\tag{**}\phi^{-1}([\gsimple{x}])\in H_x + \sum_y v^{-1}\ZZ[v^{-1}]H_y.\end{equation}
However, let's at least get the following out of the way.
\begin{thm}[{cf.\ \cite[Thm.\ 4.4]{So08}}]\label{klequivtilt}The following two statements are equivalent:
\begin{enumerate}
\item $\phi^{-1}([\gsimple{x}]) = b(C_x)$.
\item $\phi^{-1}([\gtilting{x}])=C_x$.
\end{enumerate}
\end{thm}

\begin{proof}
The Grothendieck group $K_0(\cO_0^{\ZZ})$ comes with a symmetric $\ZZ[v,v^{-1}]$-bilinear form given by
\[ \langle [M], [N] \rangle = \sum_{i,n} (-1)^i \dim\,\Ext^i(M, \DD N \langle n \rangle)v^{-n}, \]
for $M,N\in\cO_0^{\ZZ}$. With respect to this form the $[\gprojective{x}]$ and the $[\gsimple{x}]$ are dual bases, whereas the $[\gVerma{x}]$ form an orthonormal basis. Via $\phi$ this descends to the $\ZZ[v,v^{-1}]$-bilinear form on $\cH$ defined by $\langle H_x ,H_y\rangle =\delta_{x,y}$. Let $\{P_x\}_{x\in W}$ be the basis dual to $\{b(C_x)\}_{x\in W}$ in $\cH$. 
In \cite[\S3]{Virkh}, the basis dual to $\{C_x\}_{x\in W}$ is constructed combinatorially; denote this basis by $\{P'_x\}_{x\in W}$.
Then in \cite[Thm.\ 4.3]{Virkh} it is shown that
$b(C_x)H_{w_0} = P'_{xw_0}$
for all $x\in W$. Let $i\colon \cH\to\cH$ denote the ring anti-automorphism given by $i(v)=v$ and $i(H_x)=H_{x^{-1}}$. The morphisms $b,d$ and $i$ pairwise commute. Consequently, applying $i$ to $b(C_x)H_{w_0} = P'_{xw_0}$ we infer that $H_{w_0}b(C_{x^{-1}}) = P'_{w_0x^{-1}}$ or equivalently
$H_{w_0}b(C_{x}) = P'_{w_0x}$
for all $x\in W$. On the other hand, it is clear that $P'_{x} = b(P_x)$ for all $x\in W$. Thus, applying $b$ to the above, we deduce that
\[ H_{w_0}C_x = P_{w_0x} \]
for all $x\in W$.
Combining this with \eqref{definegtilt} gives the result.
\end{proof}

\begin{assumption}\label{koszul}
The ring $A$ is positively graded, i.e., $A=\bigoplus_{i\geq 0}A_i$, where $A_i$ is the homogeneous component of degree $i$. Further, the ring $A_0$ is semisimple.
\end{assumption}

\begin{remark}The above assumption is known to be true \cite[Lemma 19, Erweiterungssatz 17]{So90}, also see \cite{BGS}. However, as far as I am aware, all known proofs of this require geometric arguments.
\end{remark}

\begin{thm}\label{klconj}If Assumption \ref{koszul} holds, then \eqref{kleq} holds.
\end{thm}

\begin{proof}Let $x\in W$.
As $A$ is positively graded and the unique simple quotient of $\gVerma{x}$ (namely $\gsimple{x}$) is concentrated in degree $0$, we infer that $\gVerma{x}$ is concentrated in degrees $\geq 0$. Since $A_0$ is semisimple, the degree $0$ component of $\gVerma{x}$ is also semisimple. This forces the degree $0$ component of $\gVerma{x}$ to be $\gsimple{x}$. Thus, at the level of $K_0(\cO_0^{\ZZ})$ we have
\[ [\gVerma{x}] = \gsimple{x}+ \sum_{y < x} m_{y,x}[\gsimple{y}\langle n_{y,x}\rangle], \]
for some $m_{y,x}\in \ZZ_{\geq 0}$ and $n_{y,x}> 0$. By induction on the Bruhat order this implies 
\[ [\gsimple{x}] = [\gVerma{x}] + \sum_{y<x} m'_{y,x}[\gVerma{y}\langle n'_{y,x}\rangle] \]
for some $m'_{y,x}\in \ZZ$ and $n'_{y,x}>0$. This gives \eqref{uppertriang} which immediately yields \eqref{kleq}.
\end{proof}

\end{document}